\documentclass{amsart}

\usepackage{amssymb}
\usepackage{amsmath}
\usepackage{amsthm}
\usepackage{amsfonts}

\usepackage{a4wide}
\usepackage{longtable}
\usepackage{multirow}

\newtheorem{proposition}{Proposition}[section]
\newtheorem{theorem}{Theorem}[section]
\newtheorem{corollary}{Corollary}[theorem]

\theoremstyle{definition}
\newtheorem{definition}{Definition}[section]

\theoremstyle{remark}
\newtheorem{remark}{Remark}[section]
\newtheorem{example}{Example}[section]

\begin{document}

\title[Quadratic \(p\)-ring spaces]
{Quadratic \(p\)-ring spaces\\for counting dihedral fields}

\author{Daniel C. Mayer}
\address{Naglergasse 53\\8010 Graz\\Austria}
\email{algebraic.number.theory@algebra.at}
\urladdr{http://www.algebra.at}
\thanks{Research supported by the
Austrian Science Fund (FWF):
P 26008-N25}

\subjclass[2000]{Primary 11R29, 11R20, 11R16, 11R11; Secondary 11Y40}
\keywords{\(p\)-multiplicity, discriminant, dihedral field, cubic field, \(p\)-ring space, quadratic field}

\date{November 12, 2013}


\begin{abstract}

Let \(p\) denote an odd prime.
For all \(p\)-admissible conductors \(c\)
over a quadratic number field \(K=\mathbb{Q}(\sqrt{d})\),
\(p\)-ring spaces \(V_p(c)\) modulo \(c\)
are introduced by defining a morphism
\(\psi:\,f\mapsto V_p(f)\)
from the divisor lattice \(\mathbb{N}\) of positive integers
to the lattice \(\mathcal{S}\) of subspaces
of the direct product \(V_p\) of
the \(p\)-elementary class group \(\mathcal{C}/\mathcal{C}^p\) and
unit group \(U/U^p\) of \(K\).
Their properties admit an exact count
of all extension fields \(N\) over \(K\),
having the dihedral group of order \(2p\)
as absolute Galois group \(\mathrm{Gal}(N|\mathbb{Q})\)
and sharing a common discriminant \(d_N\) and conductor \(c\) over \(K\).
The number \(m_p(d,c)\) of these extensions is given by a formula
in terms of positions of \(p\)-ring spaces in \(\mathcal{S}\),
whose complexity increases
with the dimension of the vector space \(V_p\)
over the finite field \(\mathbb{F}_p\),
called the modified \(p\)-class rank \(\sigma_p\) of \(K\).
Up to now,
explicit multiplicity formulas for discriminants
were known for quadratic fields with
\(0\le\sigma_p\le 1\)
only.
Here,
the results are extended to
\(\sigma_p=2\),
underpinned by concrete numerical examples.

\end{abstract}

\maketitle



\section{Introduction}
\label{s:Intro}

The first class field theoretic attempt
to determine the number \(N(d_L)\)
of all non-isomorphic cubic fields \(L\vert\mathbb{Q}\)
with assigned discriminant \(d_L\)
is due to Hasse
\cite[\S 4, pp. 578--582]{Ha}.
His theorem
\cite[Satz 8, p. 581]{Ha}
reduces the problem to a system of linear congruences 
without giving a general explicit formula for \(N(d_L)\).
However, Hasse states two special formulas
\cite[I, p. 581]{Ha}
and
\cite[II, p. 582]{Ha}
which we recall in Theorems
\ref{thm:UnRmf}
and
\ref{thm:Codim0}
as formulas of codimension \(0\),
after providing adequate concepts and notation in section
\ref{s:QdrInv}.
We point out that other approaches to the enumeration problem of cubic fields
have been initiated by Berwick
\cite{Bw}
via pairs of dual quadratic fields,
by Reichardt
\cite{Re}
via Kummer theory,
and by Belabas
\cite{Be1}
via cubic forms. 

In
\cite[\S 3, p. 835]{Ma3}
we extended Hasse's theory to fields \(L\vert\mathbb{Q}\)
of arbitrary odd prime degree \(p\ge 3\)
with dihedral normal closure \(N\) of degree \(2p\).
A central role is played by the unique quadratic subfield \(K\) of \(N\),
since \(N\) is a cyclic relative extension of degree \(p\) of \(K\),
contained within the \(p\)-ring class field \(\mathrm{F}_{c,p}(K)\) of \(K\)
modulo some \(p\)-\textit{admissible conductor} \(c\) over \(K\).
The \(p\)-class rank \(\varrho_p\), the discriminant \(d\), and certain other invariants of \(K\),
called \(p\)-\textit{defects} \(\delta_p(f)\) of positive integers \(f\ge 1\) with respect to \(K\),
determine the \(p\)-\textit{multiplicity} \(m_p(d,c)\) of \(c\) with respect to \(d\),
that is the number of dihedral fields \(N\) sharing \(c\) as their common conductor over \(K\).
The corresponding formulas of codimension \(1\) for \(m_p(d,c)\) are provided by Theorems
\ref{thm:Codim1Reg}
and
\ref{thm:Codim1Irr}.

In the present article we give an essential reinterpretation of \(p\)-defects \(\delta_p(f)\)
as codimensions of \(p\)-\textit{ring spaces} \(V_p(f)\)
modulo positive integers \(f\ge 1\)
in the \(\mathbb{F}_p\)-vector space \(V_p\)
of non-trivial principal \(p\)th powers of ideals,
which is isomorphic to the direct product of
the \(p\)-elementary class group \(\mathcal{C}/\mathcal{C}^p\) and
the \(p\)-elementary unit group \(U/U^p\) of \(K\).
The introduction of this concept enables us
to derive more sophisticated formulas of higher codimension \(2\)
for multiplicities \(m_p(d,c)\) in Theorems
\ref{thm:Codim2Reg}
and
\ref{thm:Codim2Irr}.
Several corollaries and numerical examples in sections
\ref{s:CntDihFld}
and
\ref{s:MainThm}
provide useful algorithms for computing \(p\)-defects \(\delta_p(f)\) explicitly
and will hopefully remove their attribute of being \lq assez myst\'erieux\rq\ in
\cite[p. 27]{BeFv1}
and
\cite[p. 1517]{BeFv2}.



\section{Quadratic invariants}
\label{s:QdrInv}

Throughout this section
we fix a quadratic number field \(K=\mathbb{Q}(\sqrt{d})\)
with discriminant \(d\), ring of integers (maximal order) \(\mathcal{O}\), and unit group \(U\).
The group of fractional ideals of \(K\) is denoted by \(\mathcal{I}\)
and the subgroup of principal ideals by \(\mathcal{P}\).



\subsection{Modulus of definition}
\label{ss:ModOfDcl}

Since each ideal class of \(K\) contains an ideal coprime to any given integer \(f\ge 1\),
the class group of \(K\) can be defined with respect to the modulus \(f\) by the relation
\[\mathcal{C}=\mathcal{I}/\mathcal{P}=\mathcal{I}(f)\mathcal{P}/\mathcal{P}\simeq\mathcal{I}(f)/\mathcal{P}\cap\mathcal{I}(f)=\mathcal{I}(f)/\mathcal{P}(f),\]
where \(X(f)\) denotes the elements coprime to \(f\) of an arbitrary system \(X\) of ideals or numbers of \(K\).
The selection of such a \textit{modulus \(f\) of definition}
will be necessary to exclude the primes which ramify in the desired extensions \(N\vert K\)
and thus to enable an application of Artin's reciprocity law
\cite{Ar}.
Further, the choice of sufficiently large moduli \(f\)
will admit the comparison of \(p\)-ring spaces for different conductors.
To compare this ideal theoretic approach with the id\`ele theoretic version of class field theory, see
\cite[Ch. IV, \S\ 8, p. 102]{Nk1}
and
\cite[Ch. VI, \S\ 7, p. 405]{Nk2}.



\subsection{\(p\)-elementary ring class groups}
\label{ss:RingClGrp}

With an arbitrary prime \(p\ge 2\),
let the \(p\)-elementary class group of \(K\), as a subgroup of the \(p\)-class group \(\mathrm{Syl}_p(\mathcal{C})\),
be defined by \(\mathcal{C}_p=\mathcal{I}_p/\mathcal{P}\),
where \(\mathcal{I}_p=\lbrace j\in\mathcal{I}\mid j^p\in\mathcal{P}\rbrace\).
For our purpose an isomorphic representation, obtained by factoring out \(p\)th powers, is more adequate,
\[\mathcal{C}_p\simeq\mathcal{I}(f)/\mathcal{P}(f)\mathcal{I}(f)^p.\]
Similarly, the \(p\)-elementary \textit{ring class group} of \(K\) modulo an integer \(c\ge 1\) can be represented
with the aid of any integer modulus \(f\) divisible by \(c\) as
\[\mathcal{C}_{c,p}\simeq\mathcal{I}(f)/\mathcal{R}_c\mathcal{I}(f)^p,\]
where the subgroup \(\mathcal{R}_c=\lbrace\alpha\mathcal{O}\mid\alpha\in R_c\rbrace\) of \(\mathcal{P}(c)\) is called the \textit{ring} modulo \(c\) of \(K\),
with generators \(R_c=\mathbb{Q}(c)S_c\)
extending the generators \(S_c=\lbrace\alpha\in K\mid\alpha\equiv 1\pmod{c}\rbrace\) of the \textit{ray} modulo \(c\) of \(K\).

Since all groups involved are \(p\)-elementary abelian, the quotient relation
\[\mathcal{C}_p\simeq\mathcal{I}(f)/\mathcal{P}(f)\mathcal{I}(f)^p\simeq\Bigl\lbrack\mathcal{I}(f)/\mathcal{R}_c\mathcal{I}(f)^p\Bigr\rbrack\Bigm/\Bigl\lbrack\mathcal{P}(f)\mathcal{I}(f)^p/\mathcal{R}_c\mathcal{I}(f)^p\Bigr\rbrack\]
is equivalent to a direct product relation
\[\mathcal{C}_{c,p}\simeq\mathcal{I}(f)/\mathcal{R}_c\mathcal{I}(f)^p\simeq\Bigl\lbrack\mathcal{I}(f)/\mathcal{P}(f)\mathcal{I}(f)^p\Bigr\rbrack\times\Bigl\lbrack\mathcal{P}(f)\mathcal{I}(f)^p/\mathcal{R}_c\mathcal{I}(f)^p\Bigr\rbrack\]
and the ordinary \(p\)-class rank \(\varrho_p=\dim_{\mathbb{F}_p}(\mathcal{C}_p)\) is connected with
the \(p\)-ring class rank modulo \(c\), \(\varrho_{c,p}=\dim_{\mathbb{F}_p}(\mathcal{C}_{c,p})\), by the equation
\[(2.2.1) \qquad \varrho_{c,p}=\varrho_p+\Delta\varrho_{c,p},\]
where \(\Delta\varrho_{c,p}=\dim_{\mathbb{F}_p}(\mathcal{P}(f)\mathcal{I}(f)^p/\mathcal{R}_c\mathcal{I}(f)^p)\)
denotes the \textit{increment} of the \(p\)-ring class rank modulo \(c\),
which we have to analyze now.



\subsection{Principal ideal powers and prime residue classes}
\label{ss:PrncIdPwr}

Transition \(\alpha\mapsto\alpha\mathcal{O}\) from numbers to principal ideals induces an epimorphism
\[K(f)\to\mathcal{P}(f)/\mathcal{P}(f)\cap\mathcal{R}_c\mathcal{I}(f)^p\]
with kernel \(R_cI_p(f)U=R_cI_p(f)\),
where \(I_p=\lbrace\alpha\in K\mid\alpha\mathcal{O}=j^p\text{ for some }j\in\mathcal{I}\rbrace\)
denotes the group of \textit{generators of principal \(p\)th powers of ideals} of \(K\),
which contains the product \(UK^p\) as a subgroup.
Therefore, we have isomorphisms
\[\mathcal{P}(f)\mathcal{R}_c\mathcal{I}(f)^p/\mathcal{R}_c\mathcal{I}(f)^p\simeq\mathcal{P}(f)/\mathcal{P}(f)\cap\mathcal{R}_c\mathcal{I}(f)^p\simeq K(f)/R_cI_p(f)\]
and again the quotient relation of elementary abelian \(p\)-groups
\[K(f)/R_cI_p(f)\simeq\Bigl\lbrack K(f)/R_cK(f)^p\Bigr\rbrack\Bigm/\Bigl\lbrack R_cI_p(f)/R_cK(f)^p\Bigr\rbrack\]
is equivalent to a direct product relation
\[(2.3.1) \qquad K(f)/R_cK(f)^p\simeq\Bigl\lbrack K(f)/R_cI_p(f)\Bigr\rbrack\times\Bigl\lbrack R_cI_p(f)/R_cK(f)^p\Bigr\rbrack.\]
Now, \(K(f)/R_cK(f)^p\) is the \(p\)-elementary subgroup of
\[K(f)/R_c=K(f)/\mathbb{Q}(c)S_c\simeq\Bigl\lbrack K(f)/S_c\Bigr\rbrack\Bigm/\Bigl\lbrack\mathbb{Q}(c)S_c/S_c\Bigr\rbrack
\simeq\Bigl\lbrack K(f)/S_c\Bigr\rbrack\Bigm/\Bigl\lbrack\mathbb{Q}(c)/S_c\cap\mathbb{Q}(c)\Bigr\rbrack\]
and the local description of the congruence relation \(\equiv 1\pmod{c}\)
by the exact sequences given in the proof of
\cite[Thm. 4.4, p. 51]{Ma1}
yields the isomorphism
\[\Bigl\lbrack K(f)/S_c\Bigr\rbrack\Bigm/\Bigl\lbrack\mathbb{Q}(c)/S_c\cap\mathbb{Q}(c)\Bigr\rbrack\simeq U(\mathcal{O}/c\mathcal{O})\Bigm/U(\mathbb{Z}/c\mathbb{Z})\]
to the quotient of groups of prime residue classes modulo \(c\mathcal{O}\) of \(K\) and modulo \(c\mathbb{Z}\) of \(\mathbb{Q}\).

Since the prime \(p=2\) reveals an exceptional behavior
with respect to the \(p\)-rank of the quotient \(U(\mathcal{O}/c\mathcal{O})\Bigm/U(\mathbb{Z}/c\mathbb{Z})\)
and since it is irrelevant for investigating dihedral fields of degree \(2p\),
we restrict ourselves to odd primes \(p\ge 3\) in the rest of the paper.



\begin{proposition}
\label{prp:PrimResCl}

Let \(K\) be a quadratic number field,
\(p\ge 3\) an odd prime,
and \(c,f\) positive integers such that \(c\mid f\).

\begin{enumerate}

\item
The \(p\)-rank of the \(p\)-elementary subgroup of the quotient
\[U(\mathcal{O}/c\mathcal{O})\Bigm/U(\mathbb{Z}/c\mathbb{Z})\simeq K(f)/R_c\]
of prime residue class groups modulo \(c\) of \(K\) and \(\mathbb{Q}\)
is given by
\[\dim_{\mathbb{F}_p}(K(f)/R_cK(f)^p)=t+w.\]

\item
The \(p\)-rank of the \(p\)-elementary subgroup of the
prime residue class group modulo \(c\) of \(K\),
\[U(\mathcal{O}/c\mathcal{O})\simeq K(f)/S_c\]
is given by
\[\dim_{\mathbb{F}_p}(K(f)/S_cK(f)^p)=t+\tilde t+w+\tilde w.\]

\end{enumerate}

\noindent
Here,
\[t=\#\lbrace q\in\mathbb{P}\setminus\lbrace p\rbrace\mid v_q(c)\ge 1,\ q\equiv\left(\frac{d}{q}\right)\pmod{p}\rbrace\]
denotes the number of prime divisors \(q\ne p\) of the integer \(c\) having the following property:
either \(q\equiv +1\pmod{p}\) and \(q\) splits in \(K\)
or \(q\equiv -1\pmod{p}\) and \(q\) remains inert in \(K\).
The contribution of the critical prime \(p\) is also dependent on the decomposition behavior in \(K\):
\[w=
\begin{cases}
0 & \text{ if } v_p(c)=0 \\
  & \text{ or } v_p(c)=1,\ p\nmid d, \\
1 & \text{ if } v_p(c)\ge 2,\ p\nmid d \\
  & \text{ or } v_p(c)\ge 1,\ p\mid d,\ p\ge 5 \\
  & \text{ or } v_p(c)\ge 1,\ p=3,\ d\equiv +3\pmod{9} \\
  & \text{ or } v_p(c)=1,\ p=3,\ d\equiv -3\pmod{9}, \\
2 & \text{ if } v_p(c)\ge 2,\ p=3,\ d\equiv -3\pmod{9}.
\end{cases}
\]
On the other hand, the numbers
\[\tilde t=\#\lbrace q\in\mathbb{P}\setminus\lbrace p\rbrace\mid v_q(c)\ge 1,\ q\equiv +1\pmod{p}\rbrace\]
and
\[\tilde w=
\begin{cases}
0 & \text{ if } v_p(c)\le 1, \\
1 & \text{ if } v_p(c)\ge 2
\end{cases}
\]
are independent of \(K\).

\end{proposition}



\begin{proof}

For each prime divisor \(q\) of \(c\),
the local contribution of \(q\) to the \(p\)-rank of
\(U(\mathcal{O}/c\mathcal{O})\)
and
\(U(\mathcal{O}/c\mathcal{O})\Bigm/U(\mathbb{Z}/c\mathbb{Z})\)
must be determined.
All these contributions are combined additively
by means of the Chinese remainder theorem,
\(U(\mathcal{O}/\prod_{\mathfrak{Q}\mid c}\mathfrak{Q}^{v_{\mathfrak{Q}}(c)})
\simeq\prod_{\mathfrak{Q}\mid c}U(\mathcal{O}/\mathfrak{Q}^{v_{\mathfrak{Q}}(c)})\),
where the product runs over all prime ideals of \(K\) dividing \(c\).
Let \(C(m)\) denote the cyclic group of order \(m\).
For a prime \(q\ne 2\) and \(n\ge 1\), we generally have
\(U(\mathbb{Z}/q^n\mathbb{Z})\simeq C(q-1)\times C(q^{n-1})\)
with cardinality \(\varphi(q^n)=(q-1)q^{n-1}\)
and \(p\)-contribution
\(C(p^{v_p(q-1)})\) if \(q\equiv +1\pmod{p}\)
and \(C(p^{n-1})\) if \(q=p\).\\
For the structure of \(U(\mathcal{O}/q^n\mathcal{O})\), however,
we must distinguish the decomposition types of \(q\) in \(K\)
and we use
\cite{HK}.
Generally, for a prime ideal \(\mathfrak{Q}\) of \(K\)
the cardinality of \(U(\mathcal{O}/\mathfrak{Q}^n)\) is
\[\# U(\mathcal{O}/\mathfrak{Q}^n)=\varphi_K(\mathfrak{Q}^n)=(\mathrm{Norm}(\mathfrak{Q})-1)\mathrm{Norm}(\mathfrak{Q})^{n-1}\]
in terms of the generalized Euler totient function.
Let \(\Delta_n\), resp. \(\delta_n\), denote the difference between the \(p\)-ranks of two groups
\(U(\mathcal{O}/q^e\mathcal{O})\), resp. \(U(\mathcal{O}/q^e\mathcal{O})\Bigm/U(\mathbb{Z}/q^e\mathbb{Z})\),
modulo powers of \(q\) with subsequent exponents \(e=n\) and \(e=n-1\).

If \(q\) splits in \(K\), that is \(\left(\frac{d}{q}\right)=+1\),
then \(q\mathcal{O}=\mathfrak{Q}\mathfrak{Q}^\prime\) and
\[U(\mathcal{O}/q^n\mathcal{O})
\simeq U(\mathcal{O}/\mathfrak{Q}^n)\times U(\mathcal{O}/{\mathfrak{Q}^\prime}^n)
\simeq C(q-1)\times C(q^{n-1})\times C(q-1)\times C(q^{n-1})\]
with \(p\)-contribution
\(C(p^{v_p(q-1)})^2\), \(\Delta_1=2\), \(\delta_1=1\), if \(q\equiv +1\pmod{p}\)\\
and \(C(p^{n-1})^2\), \(\Delta_2=2\), \(\delta_2=1\), if \(q=p\).

If \(q\) remains inert in \(K\), i. e. \(\left(\frac{d}{q}\right)=-1\),
then \(q\mathcal{O}=\mathfrak{Q}\) and
\[U(\mathcal{O}/q^n\mathcal{O})
\simeq U(\mathcal{O}/\mathfrak{Q}^n)
\simeq C(q^2-1)\times C(q^{n-1})\times C(q^{n-1})\]
with \(p\)-contribution
\(C(p^{v_p(q-1)})\), \(\Delta_1=1\), if \(q\equiv +1\pmod{p}\),\\
however with
\(C(p^{v_p(q+1)})\), \(\Delta_1=1\), \(\delta_1=1\), if \(q\equiv -1\pmod{p}\),\\
and \(C(p^{n-1})^2\), \(\Delta_2=2\), \(\delta_2=1\), if \(q=p\).\\
The unique case where \(q=2\) gives rise to a \(p\)-contribution
occurs for \(p=3\) when \(2\) remains inert in \(K\), that is \(d\equiv 5\pmod{8}\).
Then we have irregular structures
\cite[table row 1, p. 74]{HK}
\[U(\mathbb{Z}/2^n\mathbb{Z})\simeq C(2)\times C(2^{n-2})
\text{ and }
U(\mathcal{O}/2^n\mathcal{O})
\simeq C(3)\times C(2)\times C(2^{n-2})\times C(2^{n-1})\]
without consequences for the usual result
\(\Delta_1=1\), \(\delta_1=1\) for \(2\equiv -1\pmod{3}\).

If \(q\) ramifies in \(K\), that is \(q\mid d\),
then \(q\mathcal{O}=\mathfrak{Q}^2\) and
\[U(\mathcal{O}/q^n\mathcal{O})
\simeq U(\mathcal{O}/\mathfrak{Q}^{2n})
\simeq C(q-1)\times C(q^{n-1})\times C(q^n)\]
with \(p\)-contribution
\(C(p^{v_p(q-1)})\), \(\Delta_1=1\), if \(q\equiv +1\pmod{p}\)\\
and \(C(p^{n-1})\times C(p^n)\), \(\Delta_1=1\), \(\Delta_2=1\), \(\delta_1=1\), if \(q=p\).\\
For \(q=3\) and \(d\equiv -3\pmod{9}\), we have an irregular behavior
\cite[table row 2a, p. 77]{HK}
\[U(\mathcal{O}/3^n\mathcal{O})
\simeq C(2)\times C(3)\times C(3^{n-1})\times C(3^{n-1})\]
causing the exception of two successive rank increments
\(\Delta_1=1\), \(\Delta_2=2\), \(\delta_1=1\), \(\delta_2=1\), if \(p=3\).

The following Table
\ref{tbl:DecompositionTypes}
summarizes details of the various cases.

\renewcommand{\arraystretch}{1.0}

\begin{table}[ht]
\caption{Prime decomposition types}
\label{tbl:DecompositionTypes}
\begin{center}
\begin{tabular}{|c||c|c|c|l|l|c|}
\hline
 prime \(q\) & \(\left(\frac{d}{q}\right)\) & \(q\mathcal{O}\)                    &              case & \multicolumn{2}{|c|}{rank increments}                     \\
             &                              &                                     &                   & \multicolumn{2}{|c|}{\downbracefill}                      \\
\hline
 split       &                       \(+1\) & \(\mathfrak{Q}\mathfrak{Q}^\prime\) & \(q\equiv +1(p)\) &        \(\Delta_1=2\)       &        \(\delta_1=1\)       \\
             &                              &                                     &           \(q=p\) &        \(\Delta_2=2\)       &        \(\delta_2=1\)       \\
 inert       &                       \(-1\) & \(\mathfrak{Q}\)                    & \(q\equiv +1(p)\) &        \(\Delta_1=1\)       &                             \\
             &                              &                                     & \(q\equiv -1(p)\) &        \(\Delta_1=1\)       &        \(\delta_1=1\)       \\
             &                              &                                     &           \(q=p\) &        \(\Delta_2=2\)       &        \(\delta_2=1\)       \\
 ramified    &                        \(0\) & \(\mathfrak{Q}^2\)                  & \(q\equiv +1(p)\) &        \(\Delta_1=1\)       &                             \\
             &                              &                                     &           \(q=p\) & \(\Delta_1=1,\ \Delta_2=1\) &        \(\delta_1=1\)       \\
             & \(d\equiv -3(9)\)            &                                     &         \(q=p=3\) & \(\Delta_1=1,\ \Delta_2=2\) & \(\delta_1=1,\ \delta_2=1\) \\
\hline
\end{tabular}
\end{center}
\end{table}

All the other rank increments \(\Delta_n\) and \(\delta_n\),
which have not been mentioned explicitly, are equal to zero.
Cases with \(\delta_n>0\) contribute to the counters \(t\) and \(w\).
Cases where \(\Delta_n>\delta_n\) contribute to the counters \(\tilde{t}\) and \(\tilde{w}\).

\end{proof}



Since, for any prime \(q\),
increasing the valuation \(v_q(c)\) beyond a certain threshold,
in particular beyond \(1\) for \(q\ne p\),
does not cause a further increment of the \(p\)-rank of \(K(f)/R_c\),
it is sufficient to consider \textit{nearly squarefree conductors} \(c\).

\begin{definition}
\label{dfn:AdmCnd}

\(c\) is called a \(p\)-\textit{admissible conductor} over \(K\) if
\[c=p^eq_1\cdots q_t\]
with \(t\ge 0\), pairwise distinct primes \(q_1,\ldots,q_t\in\mathbb{P}\setminus\lbrace p\rbrace\)
such that \(q_i\equiv\left(\frac{d}{q_i}\right)\pmod{p}\), and
\[e\in
\begin{cases}
\lbrace 0,2\rbrace   & \text{ if } \left(\frac{d}{p}\right)=\pm 1, \\
\lbrace 0,1\rbrace   & \text{ if } p\ge 5,\ p\mid d \\
                     & \text{ or } p=3,\ d\equiv +3\pmod{9}, \\
\lbrace 0,1,2\rbrace & \text{ if } p=3,\ d\equiv -3\pmod{9}.
\end{cases}
\]
Formally, we write \(c=q_1\cdots q_{\tau}\),
where \(\tau=t\) if \(e=0\)
and \(\tau=t+1\), \(q_{t+1}=p^e\) if \(e\ge 1\).

\end{definition}

Taking the \(p\)-ranks of both sides of equation (2.3.1)
and observing Proposition
\ref{prp:PrimResCl}(1)
gives the relation
\[(2.3.2) \qquad t+w=\Delta\varrho_{c,p}+\delta_p(c),\]
where
\(\delta_p(c)=\dim_{\mathbb{F}_p}(I_p(f)/I_p(f)\cap R_cK(f)^p)\)
is called the \(p\)-\textit{defect} of \(c\) with respect to \(K\),
similarly to
\cite[Thm. 3.1, p. 836]{Ma3}.
Consequently, \(t+w\) is the \textit{maximal possible increment} \(\Delta\varrho_{c,p}\)
of the \(p\)-ring class rank modulo \(c\)
with respect to the ordinary \(p\)-class rank of \(K\).
It is attained if and only if the \(p\)-defect \(\delta_p(c)\) vanishes.


On the other hand,
the \(p\)-defect \(\delta_p(c)\) of a \(p\)-admissible conductor \(c\) over \(K\)
is bounded from above by an estimate in terms of the prime decomposition of \(c\),
\[(2.3.3) \qquad \delta_p(c)\le t+w \qquad \textbf{(first inequality)}.\]



\subsection{\(p\)-ring spaces}
\label{ss:RingSpace}

The following isomorphisms establish a reinterpretation of the \(p\)-defect \(\delta_p(c)\) of \(c\)
which is of the greatest importance in the remaining part of the paper.
\[R_cK(f)^pI_p(f)/R_cK(f)^p\simeq I_p(f)/I_p(f)\cap R_cK(f)^p\simeq\Bigl\lbrack I_p(f)/K(f)^p\Bigr\rbrack\Bigm/\Bigl\lbrack I_p(f)\cap R_cK(f)^p/K(f)^p\Bigr\rbrack\]

\begin{definition}
\label{dfn:RingSpace}

The vector space
\(V_p=I_p/K^p=I_p(f)/K(f)^p\)
over the finite field \(\mathbb{F}_p\)
is independent of the choice of the modulus \(f\ge 1\) of definition.
It is called vector space of \textit{non-trivial generators of \(p\)th powers of ideals} of \(K\).
The subspace
\(V_p(c)=I_p(f)\cap R_cK(f)^p/K(f)^p\)
with a fixed integer \(c\ge 1\)
is independent of the choice of an integer \(f\) divisible by \(c\) as modulus of definition.
It is called the \(p\)-\textit{ring space} modulo \(c\) of \(K\)
and its codimension in \(V_p\) is exactly the \(p\)-defect of \(c\),
\[\mathrm{codim}(V_p(c))=\dim_{\mathbb{F}_p}(V_p/V_p(c))=\delta_p(c).\]

\end{definition}

\begin{remark}

The group \(I_p\), resp. \(I_p/K^p\),
is called the group of \(p\)-\textit{virtual units},
resp. the \(p\)-\textit{Selmer group}, of \(K\)
in
\cite[Dfn. 5.2.4, p. 231]{Co}.

\end{remark}



\subsection{Modified \(p\)-class rank}
\label{ss:ModifClRk}

\(p\)-defects \(\delta_p(c)\) of \(p\)-admissible conductors \(c\) over \(K\)
are bounded from above by yet another estimate involving an invariant \(\sigma_p\) of the quadratic field \(K\).
\[(2.5.1) \qquad \delta_p(c)\le\sigma_p \qquad \textbf{(second inequality)}.\]

This is due to the fact that the mapping
\(I_p\to\mathcal{I}_p/\mathcal{P}\), \(\alpha\mapsto j\),
where \(\alpha\mathcal{O}=j^p\),
is a well-defined epimorphism with kernel \(UK^p\),
which yields an isomorphism
\[\mathcal{C}_p=\mathcal{I}_p/\mathcal{P}\simeq I_p/UK^p\simeq\Bigl\lbrack I_p/K^p\Bigr\rbrack\Bigm/\Bigl\lbrack UK^p/K^p\Bigr\rbrack,\]
where \(UK^p/K^p\simeq U/U\cap K^p=U/U^p\) is the \(p\)-elementary unit group, and thus the vector space
\[V_p\simeq\mathcal{C}_p\times(U/U^p)\]
is isomorphic to the direct product of the \(p\)-elementary class group and unit group of \(K\).

\begin{definition}
\label{dfn:ModifClRk}

The dimension
\[(2.5.2) \qquad \sigma_p=\varrho_p+\dim_{\mathbb{F}_p}(U/U^p)=
\begin{cases}
\varrho_p   & \text{ if } d<-3 \\
            & \text{ or } d=-3,\ p\ge 5, \\
\varrho_p+1 & \text{ if } d>0 \\
            & \text{ or } d=-3,\ p=3
\end{cases}
\]
of the vector space \(V_p\) over \(\mathbb{F}_p\)
is called the \textit{modified \(p\)-class rank} of \(K\).

\end{definition}



Figure
\ref{fig:Defect}
visualizes the role of the \(p\)-defect \(\delta_p(c)\) of \(c\)
as the common part,
up to isomorphism,
of the \(\mathbb{F}_p\)-vector spaces
\(V_p=I_p(f)/K(f)^p\) and \(K(f)/R_cK(f)^p\)
within the radicand group \(K(f)/K(f)^p\) of Kummer theory.

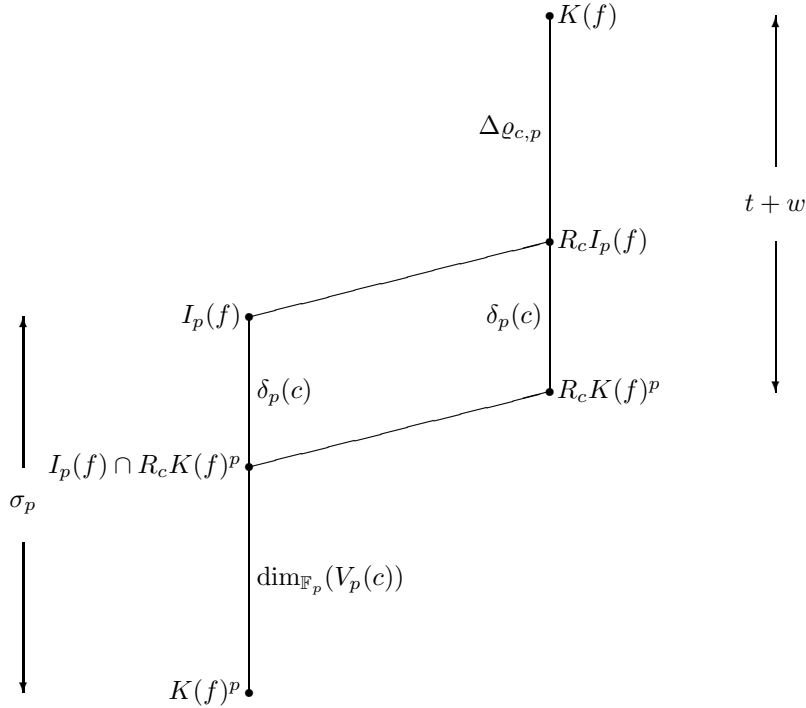
\begin{figure}[ht]
\caption{\(\delta_p(c)\) as a link between \(V_p\) and \(K(f)/R_cK(f)^p\)}
\label{fig:Defect}
\setlength{\unitlength}{1cm}
\begin{picture}(11,10)(-3,0)

\put(-3,3){\vector(0,1){2}}
\put(-3,2.5){\makebox(0,0)[cc]{\(\sigma_p\)}}
\put(-3,2){\vector(0,-1){2}}
\put(0,0){\circle*{0.1}}
\put(-0.1,0){\makebox(0,0)[rc]{\(K(f)^p\)}}
\put(0,0){\line(0,1){5}}
\put(0.1,1.5){\makebox(0,0)[lc]{\(\dim_{\mathbb{F}_p}(V_p(c))\)}}
\put(0,3){\circle*{0.1}}
\put(-0.1,3){\makebox(0,0)[rc]{\(I_p(f)\cap R_cK(f)^p\)}}
\put(0.1,4){\makebox(0,0)[lc]{\(\delta_p(c)\)}}
\put(0,5){\circle*{0.1}}
\put(-0.1,5){\makebox(0,0)[rc]{\(I_p(f)\)}}
\put(0,3){\line(4,1){4}}
\put(0,5){\line(4,1){4}}
\put(4,9){\circle*{0.1}}
\put(4.1,9){\makebox(0,0)[lc]{\(K(f)\)}}
\put(4,9){\line(0,-1){5}}
\put(3.9,7.5){\makebox(0,0)[rc]{\(\Delta\varrho_{c,p}\)}}
\put(4,6){\circle*{0.1}}
\put(4.1,6){\makebox(0,0)[lc]{\(R_cI_p(f)\)}}
\put(3.9,5){\makebox(0,0)[rc]{\(\delta_p(c)\)}}
\put(4,4){\circle*{0.1}}
\put(4.1,4){\makebox(0,0)[lc]{\(R_cK(f)^p\)}}
\put(7,7){\vector(0,1){2}}
\put(7,6.5){\makebox(0,0)[cc]{\(t+w\)}}
\put(7,6){\vector(0,-1){2}}

\end{picture}
\end{figure}



\subsection{\(p\)-ring class rank}
\label{ss:RingClRank}

\begin{theorem}
\label{prp:RingClRank}

Let \(K\) be a quadratic number field
having \(p\)-class rank \(\varrho_p\)
and modified \(p\)-class rank \(\sigma_p\),
\(p\ge 3\) an odd prime,
and \(c,f\) positive integers such that \(c\mid f\).

\begin{enumerate}

\item
The \(p\)-ring class rank modulo \(c\) is given by
\[\varrho_{c,p}=\dim_{\mathbb{F}_p}(\mathcal{I}(f)/\mathcal{R}_c\mathcal{I}(f)^p)=\varrho_p+t+w-\delta_p(c),\]
where \(\delta_p(c)\le\min\lbrace t+w,\sigma_p\rbrace\).

\item
The \(p\)-ray class rank modulo \(c\) (for general ray class groups see
\cite[p. 163]{Co})
is given by
\[\tilde\varrho_{c,p}=\dim_{\mathbb{F}_p}(\mathcal{I}(f)/\mathcal{S}_c\mathcal{I}(f)^p)=\varrho_p+t+\tilde t+w+\tilde w-\tilde\delta_p(c),\]
where \(\mathcal{S}_c=\lbrace\alpha\mathcal{O}\mid\alpha\in S_c\rbrace\subset\mathcal{R}_c\)
and \(\tilde\delta_p(c)\le\min\lbrace t+\tilde t+w+\tilde w,\sigma_p\rbrace\).

\end{enumerate}

\noindent
The numbers \(t\), \(\tilde t\), \(w\), \(\tilde w\) are defined as in Proposition
\ref{prp:PrimResCl}.
The \(p\)-\textit{ray defect} modulo \(c\) is defined by
\(\tilde\delta_p(c)=\dim_{\mathbb{F}_p}(I_p(f)/I_p(f)\cap S_cK(f)^p)\)
and satisfies
\(\delta_p(c)\le\tilde\delta_p(c)\).

\end{theorem}

\begin{proof}

Combining equations (2.2.1) and (2.3.2), we get
\(\varrho_{c,p}=\varrho_p+\Delta\varrho_{c,p}=\varrho_p+t+w-\delta_p(c)\).
The formula for the \(p\)-ray class rank \(\tilde\varrho_{c,p}\) modulo \(c\)
is a consequence of Proposition
\ref{prp:PrimResCl}(2).

\end{proof}



\section{Counting dihedral fields}
\label{s:CntDihFld}

To point out the significance of our main theorems in section
\ref{s:MainThm},
we first recall the formulas for \(p\)-multiplicities \(m_p(d,c)\)
which were known up to now
\cite{Ha,Ma3}.
It is convenient to distinguish
several types of conductors \(c\),
according to the complexity of the corresponding formulas.

\begin{definition}
\label{dfn:TypesOfCnd}

With respect to the critical prime \(p\),
a \(p\)-admissible conductor \(c=p^eq_1\cdots q_t\) over \(K=\mathbb{Q}(\sqrt{d})\)
is called \textit{irregular},
if \(p=3\), \(e=2\), and \(d\equiv -3\pmod{9}\).
Otherwise, \(c\) is called \textit{regular}.
With respect to its \(p\)-defect,
a \(p\)-admissible conductor \(c\)
is called \textit{free} if \(\delta_p(c)=0\)
(called \lq good\rq\ or \lq nice\rq\ in
\cite{Ma3})
and is called \textit{restrictive} if \(\delta_p(c)\ge 1\)
(called \lq bad\rq\ in
\cite{Ma3}).

\end{definition}

\begin{remark}

For irregular conductors \(c=3^2q_1\cdots q_t\),
we have to distinguish carefully the cases
\(\delta_3(3)=0\), resp. \(\delta_3(3)=1\),
of a \textit{free}, resp. \textit{restrictive}, \textit{prime divisor} \(3\mid c\),
which give rise to completely different multiplicity formulas.

\end{remark}



\subsection{Unramified extensions}
\label{ss:UnRmf}

The first formula (3.1.1) gives
the number of all \textit{unramified} dihedral fields \(N\) of relative degree \(p\) over \(K\)
with conductor \(c=1\)
and generalizes Hasse's formula
\cite[I, p. 581]{Ha}.



\begin{theorem}
\label{thm:UnRmf}

Let \(p\ge 3\) be an odd prime
and \(K=\mathbb{Q}(\sqrt{d})\) be a quadratic number field
with ordinary \(p\)-class rank \(\varrho=\varrho_p\ge 1\)
and discriminant \(d\).
Then \(c=1\) is a free regular \(p\)-admissible conductor over \(K\)
with \(p\)-defect \(\delta_p(1)=0\)
and the \(p\)-multiplicity of \(c=1\) with respect to \(d\) is given by
\[(3.1.1)\qquad m_p(d,1)=\frac{p^{\varrho}-1}{p-1}.\]

\end{theorem}

\begin{proof}

See \cite[Cor. 3.1, p. 838]{Ma3}.

\end{proof}



\begin{example}
\label{xpl:UnRmf}

In \textit{statistical} algebraic number theory
we are interested in solving two central problems.
Firstly, to find the minimal absolute value of the discriminant
of number fields with a certain property.
Secondly, to determine the exact number of these fields
up to a given bound for the absolute value of the discriminant.
For asymptotic results, see
\cite[\S\ 6, p. 83]{Kl}.

\renewcommand{\arraystretch}{1.0}

\begin{table}[ht]
\caption{Minimal discriminants of quadratic fields with various \(p\)-class ranks}
\label{tbl:UnRmfMin}
\begin{center}
\begin{tabular}{|c||r|r|r|r|r|r|}
\hline
 \(\varrho_p\) & \multicolumn{3}{|c|}{\(p=3\)}                                     & \multicolumn{3}{|c|}{\(p=5\)}                        \\
               & \multicolumn{3}{|c|}{\downbracefill}                              & \multicolumn{3}{|c|}{\downbracefill}                 \\
               & \(m_3(d,1)\) &            \(d<0\)         &         \(d>0\)       & \(m_5(d,1)\) &       \(d<0\)     &       \(d>0\)     \\
\hline
 \(0\)         &        \(0\) &             \(-3\)         &           \(5\)       &        \(0\) &        \(-3\)     &         \(5\)     \\
 \(1\)         &        \(1\) &            \(-23\)         &         \(229\)       &        \(1\) &       \(-47\)     &       \(401\)     \\
 \(2\)         &        \(4\) &          \(-3\,299\)       &       \(32\,099\)     &        \(6\) &    \(-11\,199\)   &    \(244\,641\)   \\
 \(3\)         &       \(13\) &       \(-3\,321\,607\)     &    \(39\,345\,017\)   &       \(31\) & \(-11\,203\,620\) & \(999\,790\,597\) \\
 \(4\)         &       \(40\) &     \(-653\,329\,427\)     & \(58\,343\,207\,081\) &      \(156\) &                   &                   \\
 \(5\)         &      \(121\) & \(-5\,393\,946\,914\,743\) &                       &      \(781\) &                   &                   \\
\hline
\end{tabular}
\end{center}
\end{table}

\begin{enumerate}

\item
Discriminants \(d\) with smallest absolute value
of quadratic fields \(K\) having a fixed \(p\)-class rank \(0\le \varrho_p\le 3\)
are given for \(p\in\lbrace 3,5\rbrace\) in Table
\ref{tbl:UnRmfMin}
together with the number \(m_p(d,1)\)
of unramified cyclic extensions \(N\vert K\) of relative degree \(p\).
Information about \(\varrho_p\ge 3\) has been taken from
\cite[Tbl. 6, p. 90]{Bu},
\cite{Be3},
\cite[Tbl. 7 and 9, pp. 472--473]{Js},
and
\cite{JRW}.

\renewcommand{\arraystretch}{1.0}

\begin{table}[ht]
\caption{Frequency of quadratic fields with various \(p\)-class ranks for \(-10^6<d<0\)}
\label{tbl:UnRmfCnt}
\begin{center}
\begin{tabular}{|c||c|r|c|r|c|r|}
\hline
\(\varrho_p\) & \multicolumn{2}{|c|}{\(p=3\)}        & \multicolumn{2}{|c|}{\(p=5\)}        & \multicolumn{2}{|c|}{\(p=7\)}        \\
              & \multicolumn{2}{|c|}{\downbracefill} & \multicolumn{2}{|c|}{\downbracefill} & \multicolumn{2}{|c|}{\downbracefill} \\
              & \(m_3(d,1)\) & \(\#\)                & \(m_5(d,1)\) & \(\#\)                & \(m_7(d,1)\) & \(\#\)                \\
\hline
 \(0\)        &        \(0\) & \(182\,323\)          &        \(0\) & \(234\,205\)          &        \(0\) & \(256\,276\)          \\
 \(1\)        &        \(1\) & \(118\,455\)          &        \(1\) &  \(69\,365\)          &        \(1\) &  \(47\,595\)          \\
 \(2\)        &        \(4\) &   \(3\,190\)          &        \(6\) &    \(398\)            &        \(8\) &     \(97\)            \\
\hline
\end{tabular}
\end{center}
\end{table}

\item
Exact counts
of complex quadratic fields \(K\)
with discriminant \(-10^6<d<0\)
having a fixed \(p\)-class rank \(0\le \varrho_p\le 2\)
are given for \(p\in\lbrace 3,5,7\rbrace\) in Table
\ref{tbl:UnRmfCnt}
together with the number \(m=m_p(d,1)\)
of members of multiplets \((N_i)_{1\le i\le m}\)
of unramified cyclic extensions \(N_i\vert K\) of degree \(p\)
\cite[p. 77]{Ma},
\cite[p. 843 ff]{Ma3}.
They will be used further in Example
\ref{xpl:Codim0}.

\end{enumerate}

\end{example}



\subsection{Free conductors}
\label{ss:FreeCond}

The following formula (3.2.1) of codimension \(0\)
enumerates the ramified dihedral fields \(N\) of relative degree \(p\) over \(K\)
with an arbitrary \textit{free} conductor \(c>1\)
and simultaneously generalizes the two formulas 
\cite[II, p. 582]{Ha}
given by Hasse.



\begin{theorem}
\label{thm:Codim0}

Let \(p\ge 3\) be an odd prime
and \(K=\mathbb{Q}(\sqrt{d})\) be a quadratic number field
with arbitrary modified \(p\)-class rank \(\sigma_p\ge 0\)
and discriminant \(d\).
Assume that \(c=p^eq_1\cdots q_t\) is a free \(p\)-admissible conductor over \(K\)
with \(p\)-defect \(\delta_p(c)=0\).
Then \(c\) is either regular or irregular
with \(p=3\), \(d\equiv -3\pmod{9}\), \(e=2\), and free prime divisor \(3\), \(\delta_3(3)=0\),
and the \(p\)-multiplicity of \(c\) with respect to \(d\) is given by
\[(3.2.1)\qquad m_p(d,c)=p^{\varrho+\omega}(p-1)^{\tau-1},\]
where \(\varrho=\varrho_p\) denotes the ordinary \(p\)-class rank of \(K\),
\(\omega=0\) if \(c\) is regular,
and \(\omega=1\) otherwise.

\end{theorem}

\begin{proof}

This statement unifies the formula for a free regular conductor
\cite[Cor. 3.2, 1, p. 839]{Ma3},
and the formula for a free irregular conductor
\cite[Cor. 3.3, case 1, p. 841]{Ma3}

\end{proof}



It is illuminating to add the following Corollary
which, though being almost trivial, provides a powerful method
for constructing ramified cyclic extensions \(N\)
with relative prime degree \(p\ge 3\)
of complex quadratic fields \(K\)
having ordinary \(p\)-class rank \(\varrho_p=0\)
by simply selecting \(p\)-admissible conductors over \(K\).
The only exception which must be excluded is
the construction of the normal closures of pure cubic fields
over Eisenstein's cyclotomic quadratic field \(K=\mathbb{Q}(\sqrt{-3})\)
for \(p=3\).

\begin{corollary}
\label{cor:Codim0}

Let \(p\ge 3\) be an odd prime
and \(K=\mathbb{Q}(\sqrt{d})\) be a complex quadratic number field
with modified \(p\)-class rank \(\sigma_p=0\)
and discriminant \(d<0\).
Assume that \(c=p^eq_1\cdots q_t\) is any \(p\)-admissible conductor over \(K\).
Then \(c\) is free with \(p\)-defect \(\delta_p(c)=0\)
and the \(p\)-multiplicity \(m_p(d,c)\) of \(c\) with respect to \(d\) is given by
formula (3.2.1) in Theorem
\ref{thm:Codim0}.

\end{corollary}

\begin{proof}

Since \(\delta_p(c)\le\sigma_p=0\) by the second inequality (2.5.1),
any \(p\)-admissible conductor over \(K\) must be free
and we can apply Theorem
\ref{thm:Codim0}.

\end{proof}



\begin{example}
\label{xpl:Codim0}

By means of Corollary
\ref{cor:Codim0}
we can determine exact counts of families of complex dihedral fields \(L\)
with multiplicities \(m_p(d,c)\in\lbrace 1,p-1,p+1\rbrace\)
up to a given bound of the root discriminant
\(\lvert d_L\rvert^{\frac{2}{p-1}}=c^2\lvert d\rvert<B\)
\cite[p. 843]{Ma3}.

\begin{enumerate}

\item
For \(B=10^6\), we have to consider prime (or prime power) conductors
\(c<\sqrt{\frac{B}{\lvert d\rvert}}\), that is,
\(c\le 577\) for \(p\ge 5\)
(actually \(c\le 571\) for \(p=5\) and \(c\le 547\) for \(p=7\))
with minimal \(\lvert d\rvert=3\), and
\(c<500\) (actually \(c\le 491\)) for \(p=3\) with minimal \(\lvert d\rvert=4\).
For each of these fixed prime (power) conductors \(c\),
we count the fundamental discriminants \(\lvert d\rvert<\frac{B}{c^2}\)
of complex quadratic fields of \(p\)-rank zero,
for which \(c\) is \(p\)-admissible.
The results are added and shown in Table
\ref{tbl:Codim0},
where the additional contributions by unramified extensions with \(c=1\)
(Table
\ref{tbl:UnRmfCnt})
and by pure cubic fields with \(d=-3\) for \(p=3\)
\cite[pp. 843--846]{Ma3}
must be taken into consideration.
The count of \(149\,204\) cubic singulets revealed
the omission of \(495\) complex cubic fields of multiplicity \(1\) in
\cite[Table 5.2, p. 321]{FgWl},
which has been corrected later in
\cite{Wl}.
These \(149\,204\) cubic fields form dominant \(81.79\%\) among a total of \(182\,417\).

\renewcommand{\arraystretch}{1.0}

\begin{table}[ht]
\caption{Count of various multiplets of dihedral fields with \(-10^6<c^2d<0\)}
\label{tbl:Codim0}
\begin{center}
\begin{tabular}{|c|r|r|r|c|r|r|c|r|r|}
\hline
 \multicolumn{4}{|c|}{\(p=3\)}                        & \multicolumn{3}{|c|}{\(p=5\)}           & \multicolumn{3}{|c|}{\(p=7\)}            \\
 \multicolumn{4}{|c|}{\downbracefill}                 & \multicolumn{3}{|c|}{\downbracefill}    & \multicolumn{3}{|c|}{\downbracefill}     \\
 \(m_3(d,c)\) & \(c=1\)      & \(c>1\)     & \(d=-3\) & \(m_5(d,c)\) &   \(c=1\)   &  \(c>1\)   & \(m_7(d,c)\) &   \(c=1\)   &   \(c>1\)   \\
\hline
        \(1\) & \(118\,455\) & \(30\,559\) &  \(190\) &        \(1\) & \(69\,365\) & \(3\,887\) &        \(1\) & \(47\,595\) &  \(2\,012\) \\
        \(2\) &        \(0\) &  \(1\,639\) &   \(44\) &        \(4\) &       \(0\) &     \(20\) &        \(6\) &       \(0\) &       \(7\) \\
        \(4\) &   \(3\,190\) &      \(15\) &   \(11\) &        \(6\) &     \(398\) &      \(0\) &        \(8\) &      \(97\) &       \(0\) \\
\hline
\end{tabular}
\end{center}
\end{table}

\item
The \(7\) sextuplets of complex cubic fields \(L\) with discriminant \(d_L=c^2d\) in Table
\ref{tbl:FakeQuintuplets}
have erroneously been announced as quintuplets in
\cite[Table 5.2, p. 321]{FgWl}.
The discriminants of their complex quadratic base fields \(K\)
satisfy the simultaneous congruences
\(d\equiv 5\pmod{8}\) and \(d\equiv -3\pmod{9}\),
thus enabling irregular conductors \(c\) divisible by \(3^2\) and also by \(2\).
In fact, they all share the common irregular conductor \(c=2\cdot 3^2=18\).
The multiplicity \(m_3(d,c)=3\cdot 2=6\) of the first \(6\) of them
\cite[Part 2, a), p. S57]{Ma3},
where \(K\) has \(3\)-class rank \(\varrho=0\),
can simply be calculated with Hasse's formula (3.2.1),
putting \(\omega=1\) and \(\tau=2\),
according to Corollary
\ref{cor:Codim0}.
The single case with \(d=-1\,371\), \(\varrho=1\), and \(\delta_3(3)=1\)
\cite[Part 2, e), p. S58]{Ma3}
is accessible by formula (3.3.2) in Theorem
\ref{thm:Codim1Irr},
where \(u=0\), \(n=2\).

\renewcommand{\arraystretch}{1.0}

\begin{table}[ht]
\caption{Seven incomplete quintuplets which are actually sextuplets \cite[Table 5.2, p. 321]{FgWl}}
\label{tbl:FakeQuintuplets}
\begin{center}
\begin{tabular}{|r|c||c|r|c|c|}
\hline
 \(d\)       & \(\varrho_3\) & \(c\)  & \(d_L=c^2d\)   & \(m_3(d,c)\) & missing regulator \(\mathrm{R}_L\) \\
\hline
    \(-291\) &         \(0\) & \(18\) &   \(-94\,284\) &        \(6\) &                          \(25.64\) \\
 \(-1\,299\) &         \(0\) & \(18\) &  \(-420\,876\) &        \(6\) &                          \(20.54\) \\
 \(-1\,659\) &         \(0\) & \(18\) &  \(-537\,516\) &        \(6\) &                          \(24.58\) \\
 \(-1\,947\) &         \(0\) & \(18\) &  \(-630\,828\) &        \(6\) &                          \(38.72\) \\
 \(-2\,307\) &         \(0\) & \(18\) &  \(-747\,468\) &        \(6\) &                          \(17.96\) \\
 \(-2\,667\) &         \(0\) & \(18\) &  \(-864\,108\) &        \(6\) &                          \(31.18\) \\
\hline
 \(-1\,371\) &         \(1\) & \(18\) &  \(-444\,204\) &        \(6\) &                          \(25.93\) \\
\hline
\end{tabular}
\end{center}
\end{table}

\end{enumerate}

\end{example}



\subsection{Restrictive conductors}
\label{ss:RstrCond}

Entirely new phenomena arise
for quadratic base fields \(K\)
with modified \(p\)-class rank \(\sigma_p=1\).
Hasse's formulas
\cite[I-II, pp. 581--582]{Ha}
can be applied to free conductors \(c\) over \(K\) only.
For \textit{restrictive} conductors with \(p\)-defect \(\delta_p(c)=1\),
however, our extensions
\cite[Cor. 3.2--3, pp. 839--841]{Ma3}
become inevitable.
The following formulas (3.3.1) and (3.3.2) of codimension \(1\)
enumerate the ramified dihedral fields \(N\) of relative degree \(p\) over \(K\)
with conductor \(c>1\) having \(\delta_p(c)=1\).



\begin{theorem}
\label{thm:Codim1Reg}

Let \(p\ge 3\) be an odd prime
and \(K=\mathbb{Q}(\sqrt{d})\) be a quadratic number field
with modified \(p\)-class rank \(\sigma_p\ge 1\)
and discriminant \(d\).
Assume that \(c=p^eq_1\cdots q_t\) is a restrictive \(p\)-admissible conductor over \(K\)
with \(p\)-defect \(\delta_p(c)=1\).
If \(c\) is either regular or irregular
with \(p=3\), \(d\equiv -3\pmod{9}\), \(e=2\), and free prime divisor \(3\), \(\delta_3(3)=0\),
then the \(p\)-multiplicity of \(c\) with respect to \(d\) is given by
\[(3.3.1)\qquad m_p(d,c)=p^{\varrho+\omega}(p-1)^u\cdot\frac{1}{p}\left((p-1)^{v-1}-(-1)^{v-1}\right),\]
where \(\varrho=\varrho_p\) denotes the ordinary \(p\)-class rank of \(K\),
\(\omega=0\) if \(c\) is regular,
and \(\omega=1\) otherwise,
\(u=\#\lbrace 1\le k\le\tau\mid V_p(q_k)=V_p\rbrace\),
and \(v=\tau-u\).

\end{theorem}

\begin{proof}

This statement unifies the formula for a restrictive regular conductor
\cite[Cor. 3.2, 2, p. 839]{Ma3},
and the two formulas for a restrictive irregular conductor with \(\delta_3(3)=0\)
\cite[Cor. 3.3, cases 2 and 3, p. 841]{Ma3}

\end{proof}

\begin{remark}

Here, the two irregular cases with free prime divisor \(3\mid c\), that is \(V_3(3)=V_3\),
are easily integrated into the regular formula (3.3.1),
because the additional factor \(3^\omega\) of the \(3\)-multiplicity is used.
In the first case, \(V_3(3^2)=V_3\), the irregular prime power divisor \(q_{t+1}=3^2\)
is counted by the occupation number \(u\ge 1\).
In the second case, \(V_3(3^2)=V_3(c)\), it contributes to the position counter \(v\ge 1\).

\end{remark}



\begin{theorem}
\label{thm:Codim1Irr}

Let \(K=\mathbb{Q}(\sqrt{d})\) be a quadratic number field
with modified \(3\)-class rank \(\sigma_3\ge 1\)
and discriminant \(d\equiv -3\pmod{9}\).
Assume that \(c=3^2q_1\cdots q_t\) is a restrictive \(3\)-admissible irregular conductor over \(K\)
with \(3\)-defect \(\delta_3(c)=1\).
Suppose the prime divisor \(3\mid c\) is restrictive with \(\delta_3(3)=1\),
whence the \(3\)-ring space \(V_3(3)\) coincides with the hyperplane \(H=V_3(c)\)
in the vector space \(V_3\).
Denote by \(n=\#\lbrace 1\le k\le\tau\mid V_3(q_k)=H\rbrace\) the occupation number of \(H\).\\
Then the \(3\)-multiplicity of \(c\) with respect to \(d\) is given by the degenerate formula
\[(3.3.2)\qquad m_3(d,c)=3^{\varrho}\cdot 2^{u+n-1},\]
where \(\varrho=\varrho_3\) denotes the \(3\)-class rank of \(K\),
\(u=\#\lbrace 1\le k\le\tau\mid V_3(q_k)=V_3\rbrace\), and \(u+n=\tau\).

\end{theorem}

\begin{proof}

This statement recalls the formula for a restrictive irregular conductor with \(\delta_3(3)=1\)
\cite[Cor. 3.3, case 4, p. 841]{Ma3}

\end{proof}

\begin{remark}

In the irregular case with restrictive prime divisor \(3\mid c\), that is \(V_3(3)=H\),
we are surprised by a \textit{degeneration}
to a formula of effectively lower codimension, similar to (3.2.1),
since the \(n\) prime divisors \(q_k\) belonging to the distinguished hyperplane \(H\)
behave quasi-free, such that actually effective values
\(u_{\mathrm{eff}}=u+n=\tau\) and \(v_{\mathrm{eff}}=v-n=0\)
replace the position counters \(u\) and \(v\).
Already in our presentation
\cite{Ma2}
we distinctly pointed out this astonishing phenomenon.
However, at that early stage it was not yet clear
that the restrictive prime divisor \(3\mid c\) with \(3\)-defect \(\delta_3(3)=1\)
is the deeper reason,
which will be confirmed by formula (4.2.1) in Theorem
\ref{thm:Codim2Irr}.

\end{remark}



In the following Examples and Corollaries
we provide techniques for dealing with quadratic fields \(K=\mathbb{Q}(\sqrt{d})\)
having modified \(p\)-class rank \(\sigma_p=1\), that is,
complex quadratic fields \(K\)
of ordinary \(p\)-class rank \(\varrho_p=1\),
the cyclotomic quadratic field \(K=\mathbb{Q}(\sqrt{-3})\)
containing the primitive \(p\)th roots of unity for \(p=3\),
and real quadratic fields \(K\)
of ordinary \(p\)-class rank \(\varrho_p=0\)
containing a fundamental unit \(\eta>1\).
In this case, it is not sufficient for the existence of an
extension \(N\) of relative degree \(p\) of \(K\)
simply to select a \(p\)-admissible conductor \(c\) for \(N\vert K\).
There is no guaranty excluding that the multiplicity \(m_p(d,c)\)
either becomes zero or takes a value different from that
expected by the formula (3.2.1) of codimension \(0\). 



\begin{example}
\label{xpl:Codim1Cmpl}

For a complex quadratic field \(K\) of discriminant \(d<0\)
and \(p\)-class rank \(\sigma_p=\varrho_p=1\),
we determine a generating principal \(p\)th power \(\alpha\) of an ideal of \(K\),
such that \(I_p/K^p=\langle\alpha\rangle\),
by means of
\cite[Algorithm, Steps 1--5, pp. 80--81]{Ma}.
That is,
we construct all reduced positive definite binary quadratic forms
\(F=(A,B,C)=AX^2+BXY+CY^2\in\mathbb{Z}\lbrack X,Y\rbrack\), \(\lvert B\rvert\le A\le C\),
of discriminant \(d_F=B^2-4AC=d\)
and we select a form \(F\) of order \(p\) with respect to composition.
Then we determine a prime \(r\) which can be represented by \(F\),
\(r=F(u,v)\) for some integers \(u,v\),
and we compute an integer solution \(x,y\) of the Diophantine equation
\(x^2-y^2d=4r^p\).
The desired generator of the vector space \(V_p=I_p/K^p\) is then given by
\(\alpha=\frac{1}{2}(x+y\sqrt{d})\).

A sufficient condition for a \(p\)-admissible conductor \(c\) over \(K\)
to be free with \(p\)-defect \(\delta_p(c)=0\) is that
\(c\mid y\), resp. \(c\mid x\),
indicated with boldface font in Table
\ref{tbl:Codim1Cmpl},
since this implies
\(\alpha\in\mathcal{O}_c\), resp. \(\alpha^2\in\mathcal{O}_c\),
and the integral generators \(R_c\cap\mathcal{O}\) of the ring modulo \(c\)
are exactly the numbers of the suborder \(\mathcal{O}_c\cap K(c)\)
which are coprime to \(c\). Consequently, the generator
\(\alpha\in\mathcal{O}_c\cap K(c)\) of \(V_p\) is contained in the \(p\)-ring space
\(V_p(c)=I_p(c)\cap R_cK(c)^p/K(c)^p\) modulo \(c\)
and thus we have codimension \(\delta_p(c)=0\).\\
Since any quadratic field \(K\)
has infinitely many free \(p\)-admissible conductors \(c\),
it is natural that the sufficient condition is not necessary in general.
In particular, for bigger conductors \(c\ge 7\)
we have to vary the represented primes \(r\)
and consequently the integer solution \((x,y)\) of the corresponding norm equation.

In Table
\ref{tbl:Codim1Cmpl}
we present the discriminants \(d\) with smallest values
of complex quadratic fields \(K\) having \(3\)-class rank \(\varrho_3=1\),
for which the prime conductors \(2\le c\le 13\), resp. the prime power conductor \(c=3^2\),
are free with \(3\)-defect \(\delta_3(c)=0\).
According to formula (3.2.1),
where we have to put \(p=3\), \(\varrho=1\), \(\omega=0\), \(\tau=1\),
the corresponding complex cubic fields \(L\) arise in triplets with discriminant \(d_L=c^2d\),
which occur in our table
\cite{Ma0}.
The unique exception is the irregular case \(c=3^2\), \(d\equiv -3\pmod{9}\),
where \(\omega=1\)
and we observe a family with \(9\) members.
It was found by Fung and Williams
\cite{FgWl},
confirmed by the above technique using quadratic forms,
and discussed in
\cite[Part 2, c), p. S57]{Ma3}.

\end{example}

\renewcommand{\arraystretch}{1.0}

\begin{table}[ht]
\caption{First occurrence of free prime conductors \(2\le c\le 13\) for \(d<-3\), \(\varrho_3=1\)}
\label{tbl:Codim1Cmpl}
\begin{center}
\begin{tabular}{|r|l|c|r|c||r|r|c|}
\hline
 \(d\)       & condition             & \(F=(A,B,C)\)  & \(r\)   &   \((x,y)\)              & \(c\)  &   \(d_L=c^2d\) & \(m_3(d,c)\) \\
\hline
    \(-307\) & \(\equiv 5\pmod{8}\)  &   \((7,1,11)\) &   \(7\) &      \((12,\mathbf{2})\) &  \(2\) &    \(-1\,228\) &        \(3\) \\
    \(-771\) & \(\equiv +3\pmod{9}\) &  \((13,3,15)\) &  \(13\) &      \((43,\mathbf{3})\) &  \(3\) &    \(-6\,939\) &        \(3\) \\
    \(-687\) & \(\equiv -3\pmod{9}\) &  \((12,9,16)\) &  \(37\) &    \((322,\mathbf{12})\) &  \(3\) &    \(-6\,183\) &        \(3\) \\
     \(-83\) & \((d/5)=-1\)          &    \((3,1,7)\) &   \(7\) &      \((\mathbf{25},3)\) &  \(5\) &    \(-2\,075\) &        \(3\) \\
     \(-59\) & \((d/7)=+1\)          &    \((3,1,5)\) &  \(19\) &    \((\mathbf{119},15)\) &  \(7\) &    \(-2\,891\) &        \(3\) \\
    \(-107\) & \(\equiv +1\pmod{3}\) &    \((3,1,9)\) &  \(13\) &      \((11,\mathbf{9})\) &  \(9\) &    \(-8\,667\) &        \(3\) \\
    \(-331\) & \(\equiv -1\pmod{3}\) &   \((5,3,17)\) &  \(19\) &      \((25,\mathbf{9})\) &  \(9\) &   \(-26\,811\) &        \(3\) \\
 \(-3\,387\) & \(\equiv -3\pmod{9}\) & \((21,15,43)\) &  \(43\) &     \((209,\mathbf{9})\) &  \(9\) &  \(-274\,347\) &        \(9\) \\
    \(-152\) & \((d/11)=-1\)         &    \((6,4,7)\) & \(137\) & \((3\,018,\mathbf{88})\) & \(11\) &   \(-18\,392\) &        \(3\) \\
     \(-87\) & \((d/13)=+1\)         &    \((4,3,6)\) & \(181\) & \((4\,846,\mathbf{52})\) & \(13\) &   \(-14\,703\) &        \(3\) \\
\hline
\end{tabular}
\end{center}
\end{table}



\noindent
The following Corollary provides a thorough proof
of our statement in
\cite[Example 1, p. 840]{Ma3}.

\begin{corollary}
\label{cor:Codim1Cyclo}

Let \(K\) denote the cyclotomic quadratic field \(\mathbb{Q}(\sqrt{-3})\)
with discriminant \(d=-3\equiv -3\pmod{9}\)
(enabling irregular \(3\)-admissible conductors),
ordinary \(p\)-class rank \(\varrho_p=0\), for any prime \(p\),
but modified \(3\)-class rank \(\sigma_3=1\).

\begin{enumerate}

\item
Every prime \(q\in\mathbb{P}\) is a \(3\)-admissible conductor over \(K\).

\item
A prime \(q\) is a free conductor over \(K\), if and only if \(q\equiv\pm1\pmod{9}\).

\item
In particular, \(q=3\) is a restrictive conductor, \(\delta_3(3)=1\), over \(K\)
and therefore the \(3\)-multiplicity \(m_3(-3,c)\)
of any irregular conductor \(c\) over \(K\)
is given by formula (3.3.2).

\item
The \(3\)-multiplicity \(m_3(-3,c)\) of a regular conductor \(c\) over \(K\)
is given by formula (3.3.1),
where \(u=\#\lbrace q\in\mathbb{P}\mid v_q(c)=1,\ q\equiv\pm1\pmod{9}\rbrace\).

\end{enumerate}

\end{corollary}



\begin{proof}

\begin{enumerate}

\item
Since the discriminant \(d=-3\) of \(K\)
satisfies \(d\equiv 5\pmod{8}\) and \(d\equiv -3\pmod{9}\),
the inert prime \(q=2\equiv -1\pmod{3}\) and the ramified prime \(q=3\)
are \(3\)-admissible conductors over \(K\).
For any other prime \(q\notin\lbrace 2,3\rbrace\), we have the Legendre symbol
\(\left(\frac{d}{q}\right)=\left(\frac{-3}{q}\right)
=\left(\frac{-1}{q}\right)\cdot\left(\frac{3}{q}\right)
=(-1)^{\frac{q-1}{2}}\cdot(-1)^{\frac{3-1}{2}\cdot\frac{q-1}{2}}\cdot\left(\frac{q}{3}\right)
=(-1)^{q-1}\cdot\left(\frac{q}{3}\right)
=\left(\frac{q}{3}\right)
\equiv q\pmod{3}\),
whence \(q\) is a \(3\)-admissible conductor over \(K\).

\item
The primitive third roots of unity are given by
\(\zeta=\frac{1}{2}(-1+\sqrt{-3})\) and \(\zeta^2=\frac{1}{2}(-1-\sqrt{-3})\).
They are not contained in any suborder \(\mathcal{O}_f\) with \(f>1\)
of the maximal order \(\mathcal{O}\) of \(K\),
whereas \(\zeta^3=1\in\mathcal{O}_f\) for each \(f\ge 1\).
Consequently,
the index of ring units modulo \(f\) is \((U:U\cap\mathcal{O}_f)=3\) for \(f>1\).
Now let \(q\) be a regular \(3\)-admissible prime conductor over \(K\),
that is, either \(q=3\), \(t=0\), \(w=1\), or \(q\ne 3\), \(t=1\), \(w=0\).
Since \(\varrho_3=0\), Proposition
\ref{prp:RingClRank}
gives the \(3\)-rank of the ring class group modulo \(q\) of \(K\),
\[
\varrho_{q,3}=\varrho_3+t+w-\delta_3(q)=1-\delta_3(q)=
\begin{cases}
1 & \text{ if } \delta_3(q)=0, \\
0 & \text{ if } \delta_3(q)=1.
\end{cases}
\]
On the other hand,
the ring class number modulo \(q\) of an arbitrary number field \(K\),
\[\lvert\mathcal{C}_q\rvert=\lvert\mathcal{C}\rvert\cdot\frac{\varphi_K(q\mathcal{O})/\varphi(q)}{(U:U\cap\mathcal{O}_q)},\]
has been determined in
\cite[pp. 50--52]{Ma1}.
Since the ordinary class number of \(K=\mathbb{Q}(\sqrt{-3})\) is \(\lvert\mathcal{C}\rvert=1\)
and \((U:U\cap\mathcal{O}_q)=3\),
we obtain the following chain of equivalences
\[\delta_3(q)=0\iff\varrho_{q,3}=1\iff v_3(\lvert\mathcal{C}_q\rvert)\ge 1\iff v_3(E(q))\ge 2,\]
where \(E(q)\) denotes the quotient of the Euler totient numbers \(\varphi_K(q\mathcal{O})/\varphi(q)\).
According to the proof of Proposition
\ref{prp:PrimResCl},
\[
v_3(E(q))=
\begin{cases}
v_3(q-1) & \text{ for } q\equiv +1\pmod{3}, \\
v_3(q+1) & \text{ for } q\equiv -1\pmod{3}, \\
v_3(3)=1 & \text{ for } q=3.
\end{cases}
\]
Consequently,
\(\delta_3(q)=0\) if and only if \(v_3(q\pm 1)\ge 2\), that is \(q\equiv\pm1\pmod{9}\).

\item
By (2), we conclude that \(\delta_3(3)=1\), since \(v_3(E(3))=1<2\).

\item
This follows immediately from (2).

\end{enumerate}

\end{proof}



\begin{example}
\label{xpl:Codim1Cyclo}

The sextuplets in Table
\ref{tbl:FakeQuintuplets},
erroneously announced as quintuplets in
\cite[Table 5.2, p. 321]{FgWl},
have shown that multiplicity \(m_3(d,c)=5\)
is impossible over a complex quadratic base field \(K\)
of \(3\)-class rank \(\varrho_3=\sigma_3=0\),
since any conductor \(c\) is necessarily free.
However,
the Eisenstein cyclotomic field \(K=\mathbb{Q}(\sqrt{-3})\)
with \(\varrho_3=0\), \(\sigma_3=1\),
admits restrictive \(3\)-admissible conductors,
as we have seen in Corollary
\ref{cor:Codim1Cyclo}.
If we take the five smallest restrictive prime conductors
\(2,3,5,7,11\) over \(K\)
and form their product \(c=2\,310\),
then we have \(\varrho=0\), \(\omega=0\), \(u=0\), \(v=t+w=5\)
in formula (3.3.1) and the restrictive factor
\(\frac{1}{p}\left((p-1)^{v-1}-(-1)^{v-1}\right)\),
which is tabulated for \(p=3\) in column \(u=0\) of the table
\cite[p. 840]{Ma3},
takes the rather exotic multiplicity
\(m_3(-3,2\,310)=\frac{1}{3}\left(2^{5-1}-(-1)^{5-1}\right)=\frac{15}{3}=5\)
as its value.
In
\cite[Example 5, p. 835]{Ma3}
the cubefree normalized radicands \(n\in\lbrace 770,3\,850,7\,700,10\,780,16\,940\rbrace\),
of the corresponding \textit{actual quintuplet} of pure cubic fields \(L=\mathbb{Q}(\root 3\of{n})\)
are given. 
The value of the common discriminant \(d_L=-16\,008\,300\) is minimal
with multiplicity \(5\) among all complex cubic fields.
Since the radicands \(n\) are all of Dedekind's first kind
\(n\not\equiv\pm 1\pmod{9}\) and additionally satisfy \(3\nmid n\),
the counter \(v=4\) of the prime divisors of \(n\) in
\cite{Ma3}
is smaller than our counter \(v=5\) of prime divisors of the conductor \(c\).

\end{example}



\begin{corollary}[Regulator Quotient Criterion]
\label{cor:Codim1Real}

Let \(p\ge 3\) be an odd prime and
\(K\) be a real quadratic field
with discriminant \(d>0\) and
ordinary \(p\)-class rank \(\varrho_p=0\).
Then \(K\) has modified \(p\)-class rank \(\sigma_p=1\)
and the following characterization of
\(p\)-admissible conductors \(q\) over \(K\) holds.

\begin{enumerate}

\item
A prime \(q\), or the prime power \(q=p^2\),
is a free regular conductor over \(K\), \(\delta_p(q)=0\), if and only if
the quotient of the regulator \(\mathrm{R}(q)\) of the suborder \(\mathcal{O}_q\)
by the regulator \(\mathrm{R}(1)\) of the maximal order \(\mathcal{O}\) of \(K\)
satisfies the condition
\[v_p(\mathrm{R}(q)/\mathrm{R}(1))<v_p(E(q)),\]
where
\[
v_p(E(q))=
\begin{cases}
v_p(q-1) & \text{ for } q\equiv +1\pmod{p},\ \left(\frac{d}{q}\right)=+1, \\
v_p(q+1) & \text{ for } q\equiv -1\pmod{p},\ \left(\frac{d}{q}\right)=-1, \\
1        & \text{ for } q=p,\ p\mid d, \\
1        & \text{ for } q=p^2,\ \left(\frac{d}{p}\right)=\pm 1.
\end{cases}
\]

\item
If \(p=3\) and
\(K\) has discriminant \(d\equiv -3\pmod{9}\)
and thus enables irregular \(3\)-admissible conductors,
then the following criteria hold for the irregular prime power conductor \(q=3^2=9\).
\[\delta_3(9)=0 \iff v_3(\mathrm{R}(9)/\mathrm{R}(1))=0,\]
\[\delta_3(9)=1 \iff v_3(\mathrm{R}(9)/\mathrm{R}(1))=1.\]

\end{enumerate}

\end{corollary}



\begin{proof}

The proof uses similar techniques for an arbitrary prime \(p\ge 3\)
as the proof of Corollary
\ref{cor:Codim1Cyclo}, (2)
for \(p=3\).

\begin{enumerate}

\item
Let \(q\) be a regular \(p\)-admissible prime conductor over \(K\),
that is, either \(q=p\), \(t=0\), \(w=1\), or \(q\ne p\), \(t=1\), \(w=0\).
Since it is assumed that \(\varrho_p=0\), Proposition
\ref{prp:RingClRank}
gives the \(p\)-rank,
\[
\varrho_{q,p}=\varrho_p+t+w-\delta_p(q)=1-\delta_p(q)=
\begin{cases}
1 & \text{ if } \delta_p(q)=0, \\
0 & \text{ if } \delta_p(q)=1, \\
\end{cases}
\]
of the ring class group modulo \(q\) of \(K\).
Again, we use the formula for the ring class number modulo \(q\) of \(K\),
\[\lvert\mathcal{C}_q\rvert=\lvert\mathcal{C}\rvert\cdot\frac{\varphi_K(q\mathcal{O})/\varphi(q)}{(U:U\cap\mathcal{O}_q)},\]
which was proved in
\cite[pp. 50--52]{Ma1}.
Since the ordinary class number \(\lvert\mathcal{C}\rvert\) of \(K\) is supposed to be coprime to \(p\),
we obtain the following chain of equivalences
\[\delta_p(q)=0\iff\varrho_{q,p}=1\iff v_p(\lvert\mathcal{C}_q\rvert)\ge 1\iff
v_p(E(q))>v_p((U:U\cap\mathcal{O}_q)),\]
where \(E(q)\) denotes the quotient \(\varphi_K(q\mathcal{O})/\varphi(q)\), as before.
The index \((U:U\cap\mathcal{O}_q)\) of the unit group of the suborder \(\mathcal{O}_q\)
in the unit group of the maximal order \(\mathcal{O}\) of \(K\)
can be expressed as the quotient of the regulator \(\mathrm{R}(q)\) of \(\mathcal{O}_q\)
by the regulator \(\mathrm{R}(1)\) of \(\mathcal{O}\) resp. \(K\).
According to the proof of Proposition
\ref{prp:PrimResCl},
\[
v_p(E(q))=
\begin{cases}
v_p(q-1) & \text{ for } q\equiv +1\pmod{p},\ \left(\frac{d}{q}\right)=+1, \\
v_p(q+1) & \text{ for } q\equiv -1\pmod{p},\ \left(\frac{d}{q}\right)=-1, \\
v_p(p)=1 & \text{ for } q=p,\ p\mid d, \\
v_p(p)=1 & \text{ for } q=p^2,\ \left(\frac{d}{p}\right)=\pm 1.
\end{cases}
\]

\item
Let \(p=3\), \(d\equiv -3\pmod{9}\), and assume that
\(q=3^2=9\) is the irregular \(3\)-admissible prime power conductor over \(K\),
that is, \(t=0\), \(w=2\).
Since \(\varrho_3=0\), Proposition
\ref{prp:RingClRank}
gives the \(3\)-rank of the ring class group modulo \(9\) of \(K\),
\[
\varrho_{9,3}=\varrho_3+t+w-\delta_3(9)=2-\delta_3(9)=
\begin{cases}
2 & \text{ if } \delta_3(9)=0, \\
1 & \text{ if } \delta_3(9)=1, \\
0 & \text{ if } \delta_3(9)=2. \\
\end{cases}
\]
The last case cannot occur, since \(\delta_3(9)\le\sigma_3=1\), according to the second inequality (2.5.1).
By the formula for the ring class number modulo \(9\) of \(K\),
\[\lvert\mathcal{C}_9\rvert=\lvert\mathcal{C}\rvert\cdot\frac{\varphi_K(9\mathcal{O})/\varphi(9)}{(U:U\cap\mathcal{O}_9)},\]
where the ordinary class number \(\lvert\mathcal{C}\rvert\) of \(K\) is supposed to be coprime to \(3\),
we obtain the following two series of equivalences
\[\delta_3(9)=0\iff\varrho_{9,3}=2\iff v_3(\lvert\mathcal{C}_9\rvert)=v_3(E(9))-v_3((U:U\cap\mathcal{O}_9))=2\iff
v_3\left(\frac{\mathrm{R}(9)}{\mathrm{R}(1)}\right)=0,\]
\[\delta_3(9)=1\iff\varrho_{9,3}=1\iff v_3(\lvert\mathcal{C}_9\rvert)=v_3(E(9))-v_3((U:U\cap\mathcal{O}_9))=1\iff
v_3\left(\frac{\mathrm{R}(9)}{\mathrm{R}(1)}\right)=1,\]
where \(v_3(E(9))=v_3(\varphi_K(9\mathcal{O})/\varphi(9))=2\)
and the factor group \(U(\mathcal{O}/9\mathcal{O})\Bigm/U(\mathbb{Z}/9\mathbb{Z})\)
is a \(3\)-elementary abelian group of rank \(2\),
by the proof of Proposition
\ref{prp:PrimResCl}.

\end{enumerate}

\end{proof}



\begin{example}
\label{xpl:Codim1Real}

The regulator quotient criterion (RQC) can be implemented easily
with the aid of the very efficient continued fraction algorithm (CFA).
It has useful applications for various problems in algebraic number theory.

\begin{enumerate}

\item
First consider \(p=3\) and
fix a prime conductor, e. g. \(q=2\equiv -1\pmod{3}\).
Then we know that \(q\) is \(3\)-admissible over all
real quadratic fields \(K=\mathbb{Q}(\sqrt{d})\) with discriminant \(d\equiv 5\pmod{8}\)
where the prime \(2\) remains inert.
We use the sequence \((5,13,21,29,\ldots)\) of these discriminants \(d\)
as input of the CFA and extract the cases
with regulator quotient \(Q=\frac{\mathrm{R}(2)}{\mathrm{R}(1)}=1\)
as a subsequence \((37,101,141,\ldots)\).
With respect to these filtered discriminants,
\(q=2\) is free with \(3\)-defect \(\delta_3(2)=0\),
provided that \(K\) has \(3\)-class rank \(\varrho_3=0\).
In this way, we have produced
\cite[Table A, pp. 53--54]{Ma1},
where the index of unit groups \(Q=(U:U\cap\mathcal{O}_2)\in\lbrace 1,3\rbrace\) was investigated,
whereas in
\cite[Part I, Section 1, 1)]{Ma4}
our intention was to construct all totally real cubic fields \(L\) with discriminant \(d_L<2\cdot 10^5\).
Those with conductor \(q=2\) have discriminants \(d_L=q^2d=4d\),
starting with \(d_L\in\lbrace 148,404,564,\ldots\rbrace\).\\
It should be emphasized, that this technique fails for \(\varrho_3\ge 1\),
as shown by
\cite[Table B, pp. 55--56]{Ma1},
where \(q=2\) is restrictive for the discriminants \(d\) in the sequence \((1\,765,2\,101,\ldots)\)
and free for \((7\,053,10\,333,11\,965,\ldots)\).
In
\cite[Part III, \(\varrho_3=1\), Section 1, 1)]{Ma4}
the corresponding cubic discriminants
\(d_L\in\lbrace 28\,212,41\,332,47\,860,\ldots\rbrace\) are listed.

\item
Now let \(p=5\),
fix a quadratic base field, e.g. \(K=\mathbb{Q}(\sqrt{5})\), with \(5\)-class rank \(\varrho_5=0\),
and vary the \(5\)-admissible prime conductors \(q\) over \(K\).
For any prime \(q\equiv\pm 1\pmod{5}\), we have the Legendre symbol
\(\left(\frac{d}{q}\right)=\left(\frac{5}{q}\right)
=(-1)^{\frac{5-1}{2}\cdot\frac{q-1}{2}}\cdot\left(\frac{q}{5}\right)
=\left(\frac{q}{5}\right)
\equiv +1\pmod{5}\),
whence the \(5\)-admissible prime conductors over \(K\) are \(q=5\)
and \(q\equiv +1\pmod{10}\) only, that is the sequence \((5,11,31,41,61,\ldots)\).
We apply the RQC to determine the free prime conductors \(q\) with \(5\)-defect \(\delta_5(q)=0\)
and find that the subsequence \((211,281,421,461,\ldots)\) starts with \(q=211\),
having an associated totally real quintic field \(L\) of discriminant \(d_L=(211^2\cdot 5)^2=222\,605^2\).
Products \(c\) of the restrictive smaller prime conductors are convenient
to display the restrictive factor
\(\frac{1}{p}\left((p-1)^{v-1}-(-1)^{v-1}\right)\)
of the multiplicity \(m_p(d,c)\) in formula (3.3.1) for \(p=5\).
We have multiplets of\\
a single quintic field for \(c=5\cdot 11\), \(v=2\),\\
\(3\) fields for \(c=5\cdot 11\cdot 31\), \(v=3\),\\
\(13\) fields for \(c=5\cdot 11\cdot 31\cdot 41\), \(v=4\), and\\
\(51\) fields for \(c=5\cdot 11\cdot 31\cdot 41\cdot 61\), \(v=5\).

\item
As a counterpart to Example
\ref{xpl:Codim1Cyclo},
we present the minimal discriminant \(d_L\)
of an \textit{actual quintuplet} of totally real cubic fields \(L\).
It is due to a personal communication by Belabas,
who used his CRF algorithm 
\cite{Be1,Be2}
and PARI
\cite{PARI}
to determine 
\(d_L=13\,302\,897\,300\) with conductor \(c=3^2\cdot 2\cdot 5\cdot 7\cdot 11\),
composed by \(v=5\) restrictive prime conductors
over the real quadratic field with discriminant \(d=277\equiv +1\pmod{3}\) and \(3\)-class rank \(\varrho_3=0\).
Previously we only knew the non-minimal case
\(d_L=89\,407\,866\,420\) with \(c=3^2\cdot 2\cdot 17\cdot 19\cdot 23\) over \(d=5\equiv -1\pmod{3}\).
Of course, this shows impressively that
a small conductor \(c\) is suited better to minimize the expression \(d_L=c^2d\)
than a small discriminant \(d\).

\end{enumerate}

\end{example}



\section{Main Theorems on conductors of \(p\)-defect \(2\)}
\label{s:MainThm}

In sections
\ref{ss:UnRmf}
and
\ref{ss:FreeCond}
the \(p\)-admissible conductors \(c\) were free
having \(p\)-defect \(\delta_p(c)=0\) and \(p\)-ring space \(V_p(c)=V_p\) the entire space.
In section
\ref{ss:RstrCond}
we presented cases of restrictive conductors \(c\)
with \(p\)-defect \(\delta_p(c)=1\) and \(p\)-ring space \(V_p(c)=H\) a hyperplane.
Now we are going to extend our theory
to base fields \(K\) which admit \(p\)-admissible conductors \(c\)
of \(p\)-defect \(\delta_p(c)=2\)
having a \(p\)-ring space \(V_p(c)\) of codimension \(2\) in \(V_p\)
\cite{Ma5}.
For \(\delta_p(c)=3\) see
\cite{Ma6}.



\subsection{Regular conductors and irregular conductors with \(\delta_3(3)=0\)}
\label{ss:Codim2Reg}

\begin{theorem}
\label{thm:Codim2Reg}

Let \(p\ge 3\) be an odd prime
and \(K=\mathbb{Q}(\sqrt{d})\) be a quadratic number field
with modified \(p\)-class rank \(\sigma_p\ge 2\)
and discriminant \(d\).
Assume that \(c=p^eq_1\cdots q_t\) is a restrictive \(p\)-admissible conductor over \(K\)
with \(p\)-defect \(\delta_p(c)=2\).
Denote by \(H_i\), \(1\le i\le p+1\), the hyperplanes of the vector space \(V_p\)
which contain the \(p\)-ring space \(V_p(c)\)
and by \(n_i=\#\lbrace 1\le k\le\tau\mid V_p(q_k)=H_i\rbrace\) their occupation numbers,
for \(1\le i\le p+1\).\\
If \(c\) is either regular or irregular
with \(p=3\), \(d\equiv -3\pmod{9}\), \(e=2\), and \(\delta_3(3)=0\),
then the \(p\)-multiplicity of \(c\) with respect to \(d\) is given by
\[(4.1.1)\qquad m_p(d,c)=p^{\varrho+\omega}(p-1)^u\cdot\frac{1}{p^2}\left((p-1)^{v-1}+\sum_{i=1}^{p+1}\,(-1)^{v-n_i}(p-1)^{n_i}\right),\]
where \(\varrho=\varrho_p\) denotes the ordinary \(p\)-class rank of \(K\),
\(\omega=0\) if \(c\) is regular,
and \(\omega=1\) otherwise,
\(u=\#\lbrace 1\le k\le\tau\mid V_p(q_k)=V_p\rbrace\),
and \(v=\tau-u\).

\end{theorem}

\begin{remark}

For \(p=3\) and \(d\equiv -3\pmod{9}\),
both irregular cases of conductors \(c\) with \(\omega=1\) and free prime divisor \(3\mid c\) having \(V_3(3)=V_3\)
are integrated into the regular formula (4.1.1) without difficulty,
using the supplementary factor \(3^{\omega}\) of the \(3\)-multiplicity.

In the first case, \(V_3(9)=V_3\),
the prime power \(q_{t+1}=9\) is free and contributes to the position counter \(u\ge 1\).
In the second case of a hyperplane \(V_3(9)=H_{i_0}\), for some \(1\le i_0\le p+1\),
it is restrictive and counted by the occupation number \(n_{i_0}\ge 1\) and thus also by \(v=\sum_{i=1}^{p+1}\,n_i\ge 1\).

\end{remark}



For the actual application of formula (4.1.1), in the case \(p=3\), Table
\ref{tbl:RstrMult}
provides the values taken by the \textit{restrictive part}
\[R(v,n_1,\ldots,n_4)=\frac{1}{3^2}\left(2^{v-1}+\sum_{i=1}^{4}\,(-1)^{v-n_i}2^{n_i}\right)\]
in the \textit{trichotomy} of the \(p\)-multiplicity formula
\[m_p(d,c)=U(\varrho_p)\cdot F(u,\omega)\cdot R(v,n_1,\ldots,n_{p+1})\]
with \textit{unramified part} \(U(\varrho)=p^{\varrho}\)
and \textit{free part} \(F(u,\omega)=p^{\omega}(p-1)^u\).

\renewcommand{\arraystretch}{1.0}

\begin{table}[ht]
\caption{Restrictive multiplicity factor in dependence on positions of \(3\)-ring spaces}
\label{tbl:RstrMult}
\begin{center}
\begin{tabular}{|c||c|c|c|c||c|c|}
\hline
  \(v\) & \(n_1\) & \(n_2\) & \(n_3\) & \(n_4\) & \(\delta_3(c)\) & \(R(v,n_1,\ldots,n_4)\) \\
\hline
  \(0\) &   \(0\) &   \(0\) &   \(0\) &   \(0\) &           \(0\) &         \(\frac{1}{2}\) \\
\hline
  \(1\) &   \(1\) &   \(0\) &   \(0\) &   \(0\) &           \(1\) &                   \(0\) \\
\hline
  \(2\) &   \(1\) &   \(1\) &   \(0\) &   \(0\) &           \(2\) &                   \(0\) \\
        &   \(2\) &   \(0\) &   \(0\) &   \(0\) &           \(1\) &                   \(1\) \\
\hline
  \(3\) &   \(1\) &   \(1\) &   \(1\) &   \(0\) &           \(2\) &                   \(1\) \\
        &   \(2\) &   \(1\) &   \(0\) &   \(0\) &           \(2\) &                   \(0\) \\
        &   \(3\) &   \(0\) &   \(0\) &   \(0\) &           \(1\) &                   \(1\) \\
\hline
  \(4\) &   \(1\) &   \(1\) &   \(1\) &   \(1\) &           \(2\) &                   \(0\) \\
        &   \(2\) &   \(1\) &   \(1\) &   \(0\) &           \(2\) &                   \(1\) \\
        &   \(2\) &   \(2\) &   \(0\) &   \(0\) &           \(2\) &                   \(2\) \\
        &   \(3\) &   \(1\) &   \(0\) &   \(0\) &           \(2\) &                   \(0\) \\
        &   \(4\) &   \(0\) &   \(0\) &   \(0\) &           \(1\) &                   \(3\) \\
\hline
  \(5\) &   \(2\) &   \(1\) &   \(1\) &   \(1\) &           \(2\) &                   \(2\) \\
        &   \(2\) &   \(2\) &   \(1\) &   \(0\) &           \(2\) &                   \(1\) \\
        &   \(3\) &   \(1\) &   \(1\) &   \(0\) &           \(2\) &                   \(3\) \\
        &   \(3\) &   \(2\) &   \(0\) &   \(0\) &           \(2\) &                   \(2\) \\
        &   \(4\) &   \(1\) &   \(0\) &   \(0\) &           \(2\) &                   \(0\) \\
        &   \(5\) &   \(0\) &   \(0\) &   \(0\) &           \(1\) &                   \(5\) \\
\hline
\end{tabular}
\end{center}
\end{table}



\subsection{Irregular conductors with \(\delta_3(3)=1\)}
\label{ss:Codim2Irr}

\begin{theorem}
\label{thm:Codim2Irr}

Let \(K=\mathbb{Q}(\sqrt{d})\) be a quadratic number field
with modified \(3\)-class rank \(\sigma_3\ge 2\)
and discriminant \(d\equiv -3\pmod{9}\).
Assume that \(c=3^2q_1\cdots q_t\) is a restrictive \(3\)-admissible conductor over \(K\)
with \(3\)-defect \(\delta_3(c)=2\).
Let \(c\) be irregular
with \(\delta_3(3)=1\), whence the \(3\)-ring space \(V_3(3)\) coincides with some hyperplane \(H\) containing \(V_3(c)\)
in the vector space \(V_3\).
Denote by \(n=\#\lbrace 1\le k\le\tau\mid V_3(q_k)=H\rbrace\) the occupation number of \(H\).\\
Then the \(3\)-multiplicity of \(c\) with respect to \(d\) is given by the degenerate formula
\[(4.2.1)\qquad m_3(d,c)=3^{\varrho}\cdot 2^{u+n}\cdot\frac{1}{3}\left(2^{v-n-1}-(-1)^{v-n-1}\right),\]
where \(\varrho=\varrho_3\) denotes the \(3\)-class rank of \(K\),
\(u=\#\lbrace 1\le k\le\tau\mid V_3(q_k)=V_3\rbrace\),
and \(v=\tau-u\).

\end{theorem}

\begin{remark}

There are two cases of irregular conductors \(c=3^2q_1\cdots q_t\)
with restrictive prime divisor \(3\mid c\), that is \(\delta_3(3)=1\) and \(V_3(3)=H\).
In both cases we observe a \textit{degeneration} to a formula of effectively lower codimension,
similar to formula (3.3.1),
since the \(n\) prime divisors \(q_k\) belonging to the distinguished hyperplane \(H\) are acting quasi-free
such that actually effective values \(u_{\mathrm{eff}}=u+n\) and \(v_{\mathrm{eff}}=v-n\)
replace the position counters \(u\) and \(v\).

In the first case, \(V_3(9)=H\),
the restrictive prime power \(q_{t+1}=9\) is counted by \(n\ge 1\) and thus also by \(u_{\mathrm{eff}}\ge 1\).
In the second case of maximal \(3\)-defect \(\delta_3(9)=2\) and \(V_3(9)=V_3(c)\),
it yields a contribution to \(v_{\mathrm{eff}}\ge 1\).

\end{remark}



\section{Proof of the main theorems}
\label{s:MainProof}



\subsection{Lattice morphism \(\psi\)}
\label{ss:Morphism}

Let \(\mathbb{N}\) be the infinite divisor lattice of positive integers
and denote by \(\mathcal{S}\) the finite lattice of subspaces
of the \(\mathbb{F}_p\)-vector space \(V_p\)
of non-trivial generators of principal \(p\)th powers of ideals
in the quadratic field \(K\).



\begin{proposition}
\label{prp:Morphism}

The morphism
\(\psi:\mathbb{N}\to\mathcal{S}\), \(c\mapsto V_p(c)\),
which maps any positive integer \(c\) to the \(p\)-ring space modulo \(c\) of \(K\)
has the following properties.

\begin{enumerate}

\item
\(\psi\) maps the minimum of \(\mathbb{N}\) to the maximum of \(\mathcal{S}\),
\(V_p(1)=V_p\).

\item
\(\psi\) maps suprema in \(\mathbb{N}\) to infima in \(\mathcal{S}\), in particular
\[V_p(c_1\cdot c_2)=V_p(c_1)\cap V_p(c_2),\text{ for coprime integers }c_1,c_2.\]

\item
\(\psi\) is anti-monotonic, that is,
if \(c_1\mid c_2\), then \(V_p(c_2)\le V_p(c_1)\), for integers \(c_1,c_2\).

\item
All values of \(\psi\) are determined uniquely by
the images \(V_p(q)\) of primes \(q\in\mathbb{P}\setminus\lbrace p\rbrace\)
and by \(V_p(p)\) and \(V_p(p^2)\).\\
In particular, the \(p\)-ring space of a \(p\)-admissible conductor \(c=q_1\cdots q_{\tau}\) over \(K\)
is given by the intersection \(V_p(c)=V_p(q_1)\cap\cdots\cap V_p(q_{\tau})\).

\item
The primes \(q\in\mathbb{P}\setminus\lbrace p\rbrace\)
which are not \(p\)-admissible conductors over \(K=\mathbb{Q}(\sqrt{d})\)
and, in the case \(p\nmid d\), the critical prime \(q=p\),
are mapped to the maximum \(V_p(q)=V_p\) by \(\psi\).

\end{enumerate}

\end{proposition}



\begin{proof}

\begin{enumerate}

\item
Since \(c=1\)
is trivially characterized by the counters \(t=0\) and \(w=0\),
the first inequality (2.3.3),
\(\delta_p(1)\le t+w=0\),
shows that \(c=1\) is free and \(V_p(1)=V_p\).

\item
Suprema in the divisor lattice \(\mathbb{N}\) are least common multiples
and infima in the lattice \(\mathcal{S}\) of subspaces are intersections.
The claim can be reduced to the property
\[V_p(c_1\cdot c_2)=V_p(c_1)\cap V_p(c_2),\text{ for coprime integers }c_1,c_2.\]
Let \(f\) be any common multiple of \(c_1\) and \(c_2\).
Then we use \(f\) as a common module of definition for the \(p\)-ring spaces
\[V_p(c_1)=I_p(f)\cap R_{c_1}K(f)^p/K(f)^p,
\ V_p(c_2)=I_p(f)\cap R_{c_2}K(f)^p/K(f)^p,\]
\[\text{and }
V_p(c_1\cdot c_2)=I_p(f)\cap R_{c_1\cdot c_2}K(f)^p/K(f)^p.\]
It is sufficient to consider the coset \(\alpha K(f)^p\)
of some integral generator \(\alpha\in I_p(f)\cap\mathcal{O}\)
of a principal \(p\)th power of an ideal of \(K\).
Since \(c_1\) and \(c_2\) are coprime, we have the equivalence
\[\alpha\in\mathcal{O}(f)\cap R_{c_1\cdot c_2}=\mathcal{O}_{c_1\cdot c_2}(f)\
\iff\]
\[\alpha\in\mathcal{O}(f)\cap R_{c_1}=\mathcal{O}_{c_1}(f)
\text{ and }
\alpha\in\mathcal{O}(f)\cap R_{c_2}=\mathcal{O}_{c_2}(f),\]
in the suborders with the corresponding conductors
of the maximal order \(\mathcal{O}\) of \(K\)
\cite[Thm. 1.1, 4, p. 22]{Ma1}.

\item
It is well known that (2) implies the anti-monotony of \(\psi\).

\item
By (2), the \(p\)-ring space of a composed conductor
is the intersection of the \(p\)-ring spaces of its prime divisors.

\item
Since a prime \(q\) which is \(p\)-inadmissible over \(K\)
is characterized by the counters \(t=0\) and \(w=0\), by definition,
this is a consequence of the first inequality (2.3.3),
\(\delta_p(q)\le t+w=0\), which implies \(V_p(q)=V_p\).

\end{enumerate}

\end{proof}



\subsection{\(s\)-occupation numbers}
\label{ss:Combinatorics}

\begin{definition}
\label{dfn:Combinatorics}

Let \(c=q_1\cdots q_{\tau}\) be a \(p\)-admissible conductor
over the quadratic field \(K\)
and assume that \(0\le s\le\tau\).
Further let \(T\le V_p\) be any subspace of the
vector space \(V_p\) of principal \(p\)th powers of ideals of \(K\).
Then the \(s\)-\textit{occupation number} of \(T\)
\[n_s(T),\text{ resp. }\tilde{n}_s(T)=\#\lbrace 1\le k_1<\ldots<k_s\le\tau,\text{ resp. }\le t\mid V_p(q_{k_1}\cdots q_{k_s})=T\rbrace\]
will denote the number of
divisors \(q_{k_1}\cdots q_{k_s}\) of \textit{length} \(s\),
that is, with \(s\) prime factors, of \(c\),
whose \(p\)-ring space coincides with \(T\).
In particular, for the entire vector space \(T=V_p\),
\[u=n_1(V_p),\text{ resp. }\tilde{u}=\tilde{n}_1(V_p)=\#\lbrace 1\le k\le\tau,\text{ resp. }\le t\mid\delta_p(q_k)=0\rbrace\]
denotes the number of \textit{free prime divisors} of \(c\).

\end{definition}



\begin{proposition}
\label{prp:Combinatorics}

Under the assumptions of Definition
\ref{dfn:Combinatorics},
the \(s\)-occupation numbers \(n_s(T)\)
for the simplest types of subspaces \(T\le V_p\)
are given by the following superpositions of binomial distributions.
Similar formulas hold for \(\tilde{n}_s\) and \(\tilde{u}\).

\begin{enumerate}

\item
For the entire vector space \(T=V_p\),
\[n_s(V_p)=\binom{u}{s}.\]

\item
For any hyperplane \(T=H<V_p\),
\[n_s(H)=\binom{u+n_1(H)}{s}-\binom{u}{s}.\]
In particular, if the \(p\)-ring space of \(c\) coincides with \(H\),
then \(u+n_1(H)=\tau\).

\item
For any subspace \(T<V_p\) of codimension \(\mathrm{codim}(T)=2\),
\[n_s(T)=\binom{u+b(T)}{s}-\sum_{H>T}\,\left\lbrack\binom{u+n_1(H)}{s}-\binom{u}{s}\right\rbrack-\binom{u}{s},\]
where the sum runs over all hyperplanes \(H\) containing \(T\)
and \(b(T)=\sum_{H>T}\,n_1(H)\) denotes the \(1\)-occupation number of the \textit{bundle of hyperplanes} over \(T\).\\
In particular, if the \(p\)-ring space of \(c\) coincides with \(T\),
then \(u+b(T)=\tau\),
except in the irregular case \(p=3\), \(d\equiv -3\pmod{9}\), \(q_{t+1}=9\), \(V_3(9)=T\),
where \(u+b(T)=t\).

\end{enumerate}

\end{proposition}



\begin{proof}

According to Proposition
\ref{prp:Morphism} (4),
\(p\)-ring spaces of composed conductors are intersections
\[V_p(q_{k_1}\cdots q_{k_s})=V_p(q_{k_1})\cap\cdots\cap V_p(q_{k_s}),
\text{ for any } 0\le s\le\tau \text{ and }
1\le k_1<\ldots<k_s\le\tau.\]

\begin{enumerate}

\item
The \(s\)-occupation number of the entire space,
\[n_s(V_p)=\#\lbrace 1\le k_1<\ldots<k_s\le\tau\mid V_p(q_{k_1}\cdots q_{k_s})=V_p\rbrace,\]
arises from products of \(s\) free prime conductors, since
\[V_p(q_{k_1}\cdots q_{k_s})=V_p\
\iff\
V_p(q_{k_1})\cap\cdots\cap V_p(q_{k_s})=V_p\
\iff\]
\[V_p(q_{k_j})=V_p,
\text{ that is } \delta_p(q_{k_j})=0, \text{ for all } 1\le j\le s.\]
Therefore, \(n_s(V_p)\) is the number of subsets of \(s\) elements of the set
\(\lbrace 1\le k\le\tau\mid\delta_p(q_k)=0\rbrace\)
with cardinality \(u=n_1(V_p)\), that is the binomial coefficient \(n_s(V_p)=\binom{u}{s}\).

\item
To calculate the \(s\)-occupation number of a hyperplane \(H\),
\[n_s(H)=\#\lbrace 1\le k_1<\ldots<k_s\le\tau\mid V_p(q_{k_1}\cdots q_{k_s})=H\rbrace,\]
we analyze the condition
\[V_p(q_{k_1}\cdots q_{k_s})=H\
\iff\
V_p(q_{k_1})\cap\cdots\cap V_p(q_{k_s})=H\
\iff\]
\[V_p(q_{k_j})\in\lbrace H,V_p\rbrace, \text{ for all } 1\le j\le s,
\text{ but not }
V_p(q_{k_j})=V_p, \text{ for all } 1\le j\le s.\]
Consequently, \(n_s(H)=\binom{u+n_1(H)}{s}-\binom{u}{s}\)
is the difference of two binomial coefficients,
where \(n_1(H)\) denotes the cardinality of the set
\(\lbrace 1\le k\le\tau\mid V_p(q_k)=H\rbrace\).

\item
Denote by \(H_1,\ldots,H_{p+1}\) the hyperplane bundle over \(T\) in \(V_p\). Then
\[V_p(q_{k_1}\cdots q_{k_s})=T\
\iff\
V_p(q_{k_1})\cap\cdots\cap V_p(q_{k_s})=T\
\iff\]
\[V_p(q_{k_j})\in\lbrace H_1,\ldots,H_{p+1},V_p\rbrace, \text{ for all } 1\le j\le s,
\text{ but}\]
\[\text{neither }
V_p(q_{k_j})\in\lbrace H_i,V_p\rbrace, \text{ for all } 1\le j\le s, \text{ for any } 1\le i\le p+1,\]
\[\text{nor }
V_p(q_{k_j})=V_p, \text{ for all } 1\le j\le s.\]
Thus,
\(n_s(T)=\binom{u+b(T)}{s}-\sum_{i=1}^{p+1}\,\left\lbrack\binom{u+n_1(H_i)}{s}-\binom{u}{s}\right\rbrack-\binom{u}{s}\),
where \(b(T)=\sum_{i=1}^{p+1}\,n_1(H_i)\).

\end{enumerate}

\end{proof}



With the aid of \(s\)-occupation numbers
we simplify the \textit{inner sum} \(L_s\) over divisors of length \(s\)
in the explicit expressions for \(p\)-multiplicities
in the case of a regular conductor
\cite[Thm. 3.2, p. 838]{Ma3}
or of an irregular conductor
\cite[Thm. 3.3, p. 841]{Ma3}.

\bigskip
\noindent
Let \(K=\mathbb{Q}(\sqrt{d})\) be an arbitrary but fixed quadratic number field
with discriminant \(d=d_K\) and \(p\)-class rank \(\varrho=\varrho_p\).



\subsection{Regular conductors}
\label{ss:RegCnd}

In the case of a regular \(p\)-admissible conductor
\(c=p^eq_1\cdots q_t>1\), with
\(0\le e\le 2\), \(t\ge 0\),
the \(p\)-multiplicity of the dihedral discriminant
\(d_N=c^{2(p-1)}\cdot d^p\)
is given by
\[(5.3.1) \qquad
m_p(d,c)
=p^{\varrho}\cdot\frac{1}{p^{\delta(c)}}\cdot
\left\lbrack(p-1)^{\tau-1}+\sum_{s=0}^{\tau}\,(-1)^{\tau-s}\cdot p^s\cdot\sum_{1\le k_1<\ldots<k_s\le\tau}\,
\frac{p^{\delta(c)-\delta(q_{k_1}\cdots q_{k_s})}-1}{p-1}\right\rbrack
\]
according to
\cite[Thm. 3.2, p. 838]{Ma3}.
For an actual contribution \(e\ge 1\) by the critical prime \(p\)
we put \(\tau=t+1\) and formally \(q_{t+1}=p^e\),
whereas \(\tau=t\) for \(e=0\).
Then we always have \(c=q_1\cdots q_{\tau}\).
Further we take into consideration that the maximal \(p\)-defect
is given by \(\delta_{\max}=\delta(c)\),
where we briefly write \(\delta(c)=\delta_p(c)\).

We rearrange the terms of the inner sum \(L_s(\delta(c))\)
over divisors \(q_{\Bbbk}=q_{k_1}\cdots q_{k_s}\) of length \(s\),
using multi indices \(\Bbbk=(k_1,\ldots,k_s)\)
in the set \(\lbrack 1,\tau\rbrack_{<}^s=\lbrace\Bbbk\mid 1\le k_1<\ldots<k_s\le\tau\rbrace\)
of cardinality \(\binom{\tau}{s}\),
\[L_s(\delta(c))=\sum_{\Bbbk\in\lbrack 1,\tau\rbrack_{<}^s}\,
\frac{p^{\delta(c)-\delta(q_{\Bbbk})}-1}{p-1}.\]
Multi indices \(\Bbbk\) sharing the same value \(\delta\) of \(\delta(q_{\Bbbk})\) are collected:
\[(5.3.2) \qquad
L_s(\delta(c))=\sum_{\delta=0}^{\delta(c)}\,\frac{p^{\delta(c)-\delta}-1}{p-1}\cdot\#\lbrace\Bbbk\mid\delta(q_{\Bbbk})=\delta\rbrace
\]



\subsection{Irregular conductors}
\label{ss:IrrCnd}

For an irregular \(3\)-admissible conductor
\(c=3^2q_1\cdots q_t\ge 3^2\),
where \(d\equiv -3\pmod{9}\),
the \(3\)-multiplicity of the dihedral discriminant
\(d_N=c^4\cdot d^3\)
is given by
\[(5.4.1) \qquad
m_3(d,c)
=3^{\varrho}\cdot\sum_{s=0}^{t}\,(-1)^{t-s}\cdot 3^s\cdot\sum_{1\le k_1<\ldots<k_s\le t}\,
\frac{3^{2-\delta(3^2\cdot q_{k_1}\cdots q_{k_s})}-3^{1-\delta(3\cdot q_{k_1}\cdots q_{k_s})}}{2}
\]
according to
\cite[Thm. 3.3, p. 841]{Ma3}.
Again, terms of the inner sum over divisors of length \(s\)
will be rearranged now.
To avoid negative exponents, it is convenient
to multiply the numerator and denominator of the fractions by \(3^{\delta(c)}\):
\[L_s(\delta(c))
=\frac{1}{3^{\delta(c)}}\cdot\sum_{\Bbbk\in[1,t]_{<}^s}\,
\frac{3^{\delta(c)+2-\delta(3^2\cdot q_{\Bbbk})}-3^{\delta(c)+1-\delta(3\cdot q_{\Bbbk})}}{2}.\]
As before,
multi indices \(\Bbbk\in[1,t]_{<}^s\) sharing the same properties are collected,
taking into account that there are two possibilities
\(\delta(3^2\cdot q_{\Bbbk})=\delta(3\cdot q_{\Bbbk})\)
and
\(\delta(3^2\cdot q_{\Bbbk})=\delta(3\cdot q_{\Bbbk})+1\),
where the latter does not yield a contribution:
\[L_s(\delta(c))=\frac{1}{3^{\delta(c)}}\cdot\sum_{\delta=0}^{\delta(c)}\,
\frac{(3-1)\cdot 3^{\delta(c)+1-\delta}}{2}
\cdot\#\lbrace\Bbbk\mid\delta(3^2\cdot q_{\Bbbk})=\delta(3\cdot q_{\Bbbk})=\delta\rbrace,\]

\[(5.4.2) \qquad
L_s(\delta(c))=\frac{1}{3^{\delta(c)}}\cdot\sum_{\delta=0}^{\delta(c)}\,
3^{\delta(c)+1-\delta}\cdot\#\lbrace\Bbbk\mid\delta(3^2\cdot q_{\Bbbk})=\delta(3\cdot q_{\Bbbk})=\delta\rbrace.\]



\subsection{Inner sums \(L_s\)}
\label{ss:InnerSum}

\begin{proposition}
\label{prp:InnerSum}

Let \(c\) be a \(p\)-admissible conductor over \(K\) with \(p\)-defect \(\delta_p(c)=2\).
Denote by \((H_i)_{1\le i\le p+1}\) the bundle of hyperplanes in \(V_p\)
which contain the \(p\)-ring space \(V_p(c)\) of \(c\).

\begin{enumerate}

\item
For a regular conductor \(c\),
\[L_s(2)=\sum_{i=1}^{p+1}\,\binom{u+n_1(H_i)}{s}.\]

\item
For an irregular conductor \(c\) with \(3\)-defect \(\delta_3(c)=2\),
and \(\delta_3(3^2)=\delta_3(3)=0\),
\[L_s(2)=\frac{1}{3}\left\lbrack\binom{t}{s}+2\cdot\sum_{i=1}^{4}\,\binom{\tilde{u}+n_1(H_i)}{s}\right\rbrack.\]

\item
For an irregular conductor \(c\) with \(3\)-defect \(\delta_3(c)=2\),
\(\delta_3(3^2)=1>\delta_3(3)=0\), and distinguished hyperplane \(V_3(3^2)=H_{i_0}\),
\[L_s(2)=\frac{1}{3}\left\lbrack\binom{t}{s}+3\cdot\binom{u+\tilde{n}_1(H_{i_0})}{s}
-\sum_{i=1}^{4}\,\binom{u+\tilde{n}_1(H_i)}{s}\right\rbrack.\]

\item
For an irregular conductor \(c\) with \(3\)-defect \(\delta_3(c)=2\),
\(\delta_3(3^2)=\delta_3(3)=1\), and distinguished hyperplane \(V_3(3^2)=V_3(3)=H_{i_0}\),
\[L_s(2)=\frac{1}{3}\left\lbrack\binom{t}{s}+2\cdot\binom{u+\tilde{n}_1(H_{i_0})}{s}\right\rbrack.\]

\item
For an irregular conductor \(c\) with \(3\)-defect \(\delta_3(c)=2\), maximal restrictivity
\(\delta_3(3^2)=2>\delta_3(3)=1\), distinguished hyperplane \(V_3(3)=H_{i_0}\), and \(V_3(3^2)=V_p(c)\),
\[L_s(2)=\frac{1}{3}\left\lbrack\binom{t}{s}-\binom{u+n_1(H_{i_0})}{s}\right\rbrack.\]

\end{enumerate}

\end{proposition}



\begin{proof}

\begin{enumerate}

\item
Firstly, we investigate the inner sum \(L_s(2)\)
over partial conductors of length \(s\)
for a regular \(p\)-admissible conductor \(c\)
having \(p\)-defect \(\delta(c)=2\).
According to equation (5.3.2), we have
\[L_s(2)=\sum_{\delta=0}^{2}\,\frac{p^{2-\delta}-1}{p-1}\cdot\#\lbrace\Bbbk\mid\delta(q_{\Bbbk})=\delta\rbrace\]
\[=\frac{p^{2}-1}{p-1}\cdot\#\lbrace\Bbbk\mid\delta(q_{\Bbbk})=0\rbrace
+\frac{p-1}{p-1}\cdot\#\lbrace\Bbbk\mid\delta(q_{\Bbbk})=1\rbrace
+\frac{1-1}{p-1}\cdot\#\lbrace\Bbbk\mid\delta(q_{\Bbbk})=2\rbrace.\]
According to Prop.
\ref{prp:Combinatorics},
the cardinalities can be expressed by binomial coefficients,
\[L_s(2)
=(p+1)\cdot\binom{u}{s}
+1\cdot\sum_{i=1}^{p+1}\,\left\lbrack\binom{u+n_1(H_i)}{s}-\binom{u}{s}\right\rbrack
+0\]
\[=(p+1)\cdot\binom{u}{s}
+\sum_{i=1}^{p+1}\,\binom{u+n_1(H_i)}{s}
-(p+1)\cdot\binom{u}{s}
=\sum_{i=1}^{p+1}\,\binom{u+n_1(H_i)}{s}.\]



\item
Secondly, we consider an irregular conductor \(c\) with free critical prime divisor \(3\) and free square \(3^2\),
that is \(\omega=1\), \(\delta(3)=0\), \(\delta(3^2)=0\).
Then the two \(3\)-ring spaces \(V(3)=V\) and \(V(3^2)=V\) coincide with the entire vectorspace \(V=V_3\).

\smallskip
\noindent
For each possibility of the \(3\)-ring space \(V(q_{\Bbbk})\),
we list the subspaces \(V(3\cdot q_{\Bbbk})\) and \(V(3^2\cdot q_{\Bbbk})\) in the table of Figure
\ref{fig:FreelyIrr},
where \(H\) denotes an arbitrary hyperplane and \(L=V(c)\) denotes the limit space.

\begin{figure}[ht]
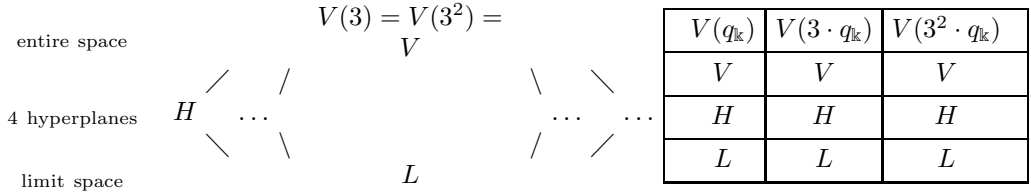

\caption{\(3\)-ring spaces \(V(f)\), \(f\mid c\), for irregular \(c\), \(\delta(3)=\delta(3^2)=0\)}
\label{fig:FreelyIrr}

$$
\vbox{
\offinterlineskip
\halign{
\strut\hfil$#$\hfil\quad  &\ \hfil$#$\hfil\ &\hfil$#$\hfil &\ \hfil$#$\hfil\ &\hfil$#$\hfil &\quad\hfil$#$\hfil\quad &\hfil$#$\hfil &\ \hfil$#$\hfil\ &\hfil$#$\hfil &\ \hfil$#$\hfil\ \cr
                          &                 &              &                 &              & V(3)=V(3^2)=           &              &                 &              &               \cr
{}^{\rm entire\ space}    &                 &              &                 &              & V                      &              &                 &              &               \cr
                          &                 & \diagup      &                 & /            &                        & \backslash   &                 & \diagdown    &               \cr
{}_{4\ {\rm hyperplanes}} & H               &              & \ldots          &              &                        &              & \ldots          &              &  \ldots       \cr
                          &                 & \diagdown    &                 & \backslash   &                        & /            &                 & \diagup      &               \cr
{}_{\rm limit\ space}     &                 &              &                 &              & L                      &              &                 &              &               \cr
}
}
\vbox{
\offinterlineskip
\halign{
 \vrule # & \strut\quad\hfill $#$\hfill \ & \vrule\vrule # & \ \hfill $#$\hfill \  & \vrule # &\ \hfill $#$\hfill \quad & \vrule # \cr
\noalign{\hrule}
          &       \omit                   & height 2pt     &                       &          &                         &          \cr
          & V(q_\Bbbk)                    &                & V(3\cdot q_\Bbbk)     &          & V(3^2\cdot q_\Bbbk)     &          \cr
          &       \omit                   & height 2pt     &                       &          &                         &          \cr
\noalign{\hrule}
          &       \omit                   & height 2pt     &                       &          &                         &          \cr
          & V                             &                & V                     &          & V                       &          \cr
          &       \omit                   & height 2pt     &                       &          &                         &          \cr
\noalign{\hrule}
          &       \omit                   & height 2pt     &                       &          &                         &          \cr
          & H                             &                & H                     &          & H                       &          \cr
          &       \omit                   & height 2pt     &                       &          &                         &          \cr
\noalign{\hrule}
          &       \omit                   & height 2pt     &                       &          &                         &          \cr
          & L                             &                & L                     &          & L                       &          \cr
          &       \omit                   & height 2pt     &                       &          &                         &          \cr
\noalign{\hrule}
}
}
$$

\end{figure}

\noindent
Thus we always have \(\delta(q_\Bbbk)=\delta(3q_\Bbbk)=\delta(9q_\Bbbk)\) and therefore,
according to equation (5.4.2) and Prop.
\ref{prp:Combinatorics},
\[L_s(2)=\frac{1}{3}\sum_{\delta=0}^2\,3^{2-\delta}\cdot\#\lbrace\Bbbk\mid\delta(9q_\Bbbk)=\delta(3q_\Bbbk)=\delta\rbrace\]
\[=\frac{1}{3}(3^2\cdot\#\lbrace\Bbbk\mid\delta(q_\Bbbk)=0\rbrace
+3\cdot\#\lbrace\Bbbk\mid\delta(q_\Bbbk)=1\rbrace
+1\cdot\#\lbrace\Bbbk\mid\delta(q_\Bbbk)=2\rbrace)\]
\[=\frac{1}{3}(9\cdot\binom{\tilde{u}}{s}
+3\cdot\sum_{i=1}^4\,\left\lbrack\binom{\tilde{u}+n_1(H_i)}{s}-\binom{\tilde{u}}{s}\right\rbrack
+\binom{t}{s}-\sum_{i=1}^4\,\left\lbrack\binom{\tilde{u}+n_1(H_i)}{s}-\binom{\tilde{u}}{s}\right\rbrack-\binom{\tilde{u}}{s})\]
\[=\frac{1}{3}(\binom{t}{s}+2\cdot\sum_{i=1}^4\,\binom{\tilde{u}+n_1(H_i)}{s}),\]
since \(\binom{\tilde{u}}{s}\) occurs \(9-12+4-1=0\) times and each \(\binom{\tilde{u}+n_1(H_i)}{s}\) occurs \(3-1=2\) times.



\item
Next, we investigate the case of an irregular conductor \(c\) with free critical prime divisor \(3\) but restrictive square \(3^2\),
that is \(\omega=1\), \(\delta(3)=0\), \(\delta(3^2)=1\).
Then \(V(3)=V\) is the entire space but \(V(3^2)\) is a hyperplane.

\smallskip
\noindent
Again, we list the subspaces \(V(3\cdot q_{\Bbbk})\) and \(V(3^2\cdot q_{\Bbbk})\) in the table of Figure
\ref{fig:TamelyIrr},
for each possibility of the \(3\)-ring space \(V(q_{\Bbbk})\),
where \(H\) denotes a hyperplane different from \(V(3^2)\) and \(L=V(c)\) denotes the limit space.

\begin{figure}[ht]
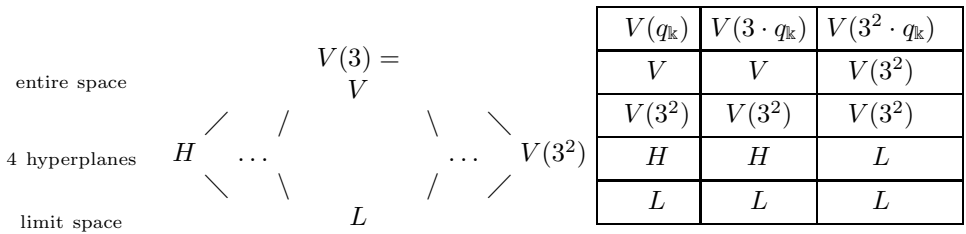

\caption{\(3\)-ring spaces \(V(f)\), \(f\mid c\), for irregular \(c\), \(\delta(3)=0\), \(\delta(3^2)=1\)}
\label{fig:TamelyIrr}

$$
\vbox{
\offinterlineskip
\halign{
\strut\hfil$#$\hfil\quad   &\ \hfil$#$\hfil\ &\hfil$#$\hfil &\ \hfil$#$\hfil\ &\hfil$#$\hfil &\quad\hfil$#$\hfil\quad &\hfil$#$\hfil &\ \hfil$#$\hfil\ &\hfil$#$\hfil &\ \hfil$#$\hfil\ \cr
                           &                 &              &                 &              & V(3)=                  &              &                 &              &               \cr
{}^{\rm entire\ space}     &                 &              &                 &              & V                      &              &                 &              &               \cr
                           &                 & \diagup      &                 & /            &                        & \backslash   &                 & \diagdown    &               \cr
{}_{4\ {\rm hyper planes}} & H               &              & \ldots          &              &                        &              & \ldots          &              & V(3^2)        \cr
                           &                 & \diagdown    &                 & \backslash   &                        & /            &                 & \diagup      &               \cr
{}_{\rm limit\ space}      &                 &              &                 &              & L                      &              &                 &              &               \cr
}
}
\vbox{
\offinterlineskip
\halign{
 \vrule # & \strut\quad\hfill $#$\hfill \ & \vrule\vrule # & \ \hfill $#$\hfill \  & \vrule # &\ \hfill $#$\hfill \quad & \vrule # \cr
\noalign{\hrule}
          &       \omit                   & height 2pt     &                       &          &                         &          \cr
          & V(q_\Bbbk)                    &                & V(3\cdot q_\Bbbk)     &          & V(3^2\cdot q_\Bbbk)     &          \cr
          &       \omit                   & height 2pt     &                       &          &                         &          \cr
\noalign{\hrule}
          &       \omit                   & height 2pt     &                       &          &                         &          \cr
          & V                             &                & V                     &          & V(3^2)                  &          \cr
          &       \omit                   & height 2pt     &                       &          &                         &          \cr
\noalign{\hrule}
          &       \omit                   & height 2pt     &                       &          &                         &          \cr
          & V(3^2)                        &                & V(3^2)                &          & V(3^2)                  &          \cr
          &       \omit                   & height 2pt     &                       &          &                         &          \cr
\noalign{\hrule}
          &       \omit                   & height 2pt     &                       &          &                         &          \cr
          & H                             &                & H                     &          & L                       &          \cr
          &       \omit                   & height 2pt     &                       &          &                         &          \cr
\noalign{\hrule}
          &       \omit                   & height 2pt     &                       &          &                         &          \cr
          & L                             &                & L                     &          & L                       &          \cr
          &       \omit                   & height 2pt     &                       &          &                         &          \cr
\noalign{\hrule}
}
}
$$

\end{figure}

\noindent
The first and third row in the table of Figure
\ref{fig:TamelyIrr}
do not give a contribution.
For the second and fourth row we have
\(\delta(q_\Bbbk)=\delta(3q_\Bbbk)=\delta(9q_\Bbbk)\) and thus,
according to equation (5.4.2) and Prop.
\ref{prp:Combinatorics},
\[L_s(2)=\frac{1}{3}\sum_{\delta=0}^2\,3^{2-\delta}\cdot\#\lbrace\Bbbk\mid\delta(9q_\Bbbk)=\delta(3q_\Bbbk)=\delta\rbrace\]
\[=\frac{1}{3}(3\cdot\#\lbrace\Bbbk\mid\delta(q_\Bbbk)=1,\ V(q_\Bbbk)=V(3^2)\rbrace
+1\cdot\#\lbrace\Bbbk\mid\delta(q_\Bbbk)=2\rbrace)\]
\[=\frac{1}{3}(3\cdot\left\lbrack\binom{u+\tilde{n}_1(V(3^2))}{s}-\binom{u}{s}\right\rbrack
+\binom{t}{s}-\sum_{i=1}^4\,\left\lbrack\binom{u+\tilde{n}_1(H_i)}{s}-\binom{u}{s}\right\rbrack-\binom{u}{s})\]
\[=\frac{1}{3}(\binom{t}{s}+3\cdot\binom{u+\tilde{n}_1(V(3^2))}{s}-\sum_{i=1}^4\,\binom{u+\tilde{n}_1(H_i)}{s}),\]
since \(\delta=0\) doesn't occur and \(\binom{u}{s}\) occurs \(-3+4-1=0\) times.



\item
Now we consider the case of an irregular conductor \(c\) with restrictive critical prime divisor \(3\) and restrictive square \(3^2\),
that is \(\omega=1\), \(\delta(3)=1\), \(\delta(3^2)=1\).
Then \(V(3)\) and \(V(3^2)\) are coinciding hyperplanes.

\smallskip
\noindent
In Figure
\ref{fig:WildlyIrr}
we construct a table of subspaces \(V(3\cdot q_\Bbbk)\) and \(V(3^2\cdot q_\Bbbk)\) for each possibility of the ring space \(V(q_\Bbbk)\),
where \(H\) denotes a hyperplane different from \(V(3)\) and \(L=V(c)\) denotes the limit space.

\begin{figure}[ht]
\caption{\(3\)-ring spaces \(V(f)\), \(f\mid c\), for irregular \(c\), \(\delta(3)=\delta(3^2)=1\)}
\label{fig:WildlyIrr}

$$
\vbox{
\offinterlineskip
\halign{
\strut\hfil$#$\hfil\quad  &\ \hfil$#$\hfil\ &\hfil$#$\hfil &\ \hfil$#$\hfil\ &\hfil$#$\hfil &\quad\hfil$#$\hfil\quad &\hfil$#$\hfil &\ \hfil$#$\hfil\ &\hfil$#$\hfil &\ \hfil$#$\hfil\ \cr
{}^{\rm entire\ space}    &                 &              &                 &              & V                      &              &                 &              &               \cr
                          &                 & \diagup      &                 & /            &                        & \backslash   &                 & \diagdown    &               \cr
{}_{4\ {\rm hyperplanes}} & H               &              & \ldots          &              &                        &              &                 &              & V(3)=V(3^2)   \cr
                          &                 & \diagdown    &                 & \backslash   &                        & /            &                 & \diagup      &               \cr
{}_{\rm limit\ space}     &                 &              &                 &              & L                      &              &                 &              &               \cr
}}
\vbox{
\offinterlineskip
\halign{
 \vrule # & \strut\quad\hfill $#$\hfill \ & \vrule\vrule # & \ \hfill $#$\hfill \  & \vrule # &\ \hfill $#$\hfill \quad & \vrule # \cr
\noalign{\hrule}
          &       \omit                   & height 2pt     &                       &          &                         &          \cr
          & V(q_\Bbbk)                    &                & V(3\cdot q_\Bbbk)     &          & V(3^2\cdot q_\Bbbk)     &          \cr
          &       \omit                   & height 2pt     &                       &          &                         &          \cr
\noalign{\hrule}
          &       \omit                   & height 2pt     &                       &          &                         &          \cr
          & V                             &                & V(3)                  &          & V(3)                    &          \cr
          &       \omit                   & height 2pt     &                       &          &                         &          \cr
\noalign{\hrule}
          &       \omit                   & height 2pt     &                       &          &                         &          \cr
          & V(3)                          &                & V(3)                  &          & V(3)                    &          \cr
          &       \omit                   & height 2pt     &                       &          &                         &          \cr
\noalign{\hrule}
          &       \omit                   & height 2pt     &                       &          &                         &          \cr
          & H                             &                & L                     &          & L                       &          \cr
          &       \omit                   & height 2pt     &                       &          &                         &          \cr
\noalign{\hrule}
          &       \omit                   & height 2pt     &                       &          &                         &          \cr
          & L                             &                & L                     &          & L                       &          \cr
          &       \omit                   & height 2pt     &                       &          &                         &          \cr
\noalign{\hrule}
}
}
$$

\end{figure}

\noindent
The first and second row in the table of Figure
\ref{fig:TamelyIrr}
give contributions to \(\delta=1\).
The third and fourth row give contributions to \(\delta=2\).
Consequently,
according to equation (5.4.2) and Prop.
\ref{prp:Combinatorics},
\[L_s(2)=\frac{1}{3}\sum_{\delta=0}^2\,3^{2-\delta}\cdot\#\lbrace\Bbbk\mid\delta(9q_\Bbbk)=\delta(3q_\Bbbk)=\delta\rbrace\]
\[=\frac{1}{3}(3\cdot\lbrack\#\lbrace\Bbbk\mid\delta(q_\Bbbk)=0\rbrace
+\#\lbrace\Bbbk\mid\delta(q_\Bbbk)=1,\ V(q_\Bbbk)=V(3)\rbrace\rbrack
+1\cdot\lbrack\#\lbrace\Bbbk\mid\delta(q_\Bbbk)=1,\ V(q_\Bbbk)\ne V(3)\rbrace
+\#\lbrace\Bbbk\mid\delta(q_\Bbbk)=2\rbrace\rbrack)\]
\[=\frac{1}{3}(3\cdot\binom{u}{s}+3\cdot\left\lbrack\binom{u+\tilde{n}_1(V(3))}{s}-\binom{u}{s}\right\rbrack
+\sum_{i\ne i_0}\,\left\lbrack\binom{u+\tilde{n}_1(H_i)}{s}-\binom{u}{s}\right\rbrack
+\binom{t}{s}-\sum_{i=1}^4\,\left\lbrack\binom{u+\tilde{n}_1(H_i)}{s}-\binom{u}{s}\right\rbrack-\binom{u}{s})\]
\[=\frac{1}{3}(\binom{t}{s}+2\cdot\binom{u+\tilde{n}_1(V(3))}{s}),\]
since \(\delta=0\) doesn't occur, \(\binom{u+\tilde{n}_1(V(3))}{s}\) occurs \(3-1=2\) times,\\
and \(\binom{u}{s}\) occurs \(3-3-3+4-1=0\) times.
Note that partially \(\delta(q_\Bbbk)<\delta(3q_\Bbbk)=\delta(9q_\Bbbk)\).



\item
Finally, we consider the case of an irregular conductor \(c\) with
restrictive critical prime divisor \(3\) and square \(3^2\) of maximal defect,
that is \(\omega=1\), \(\delta(3)=1\), \(\delta(3^2)=2\).
Then \(V(3)\) is a hyperplane but \(V(3^2)=L\) is the limit space.

\smallskip
\noindent
As before, we construct a table of subspaces \(V(3\cdot q_\Bbbk)\) and \(V(3^2\cdot q_\Bbbk)\)
for each possibility of the \(3\)-ring space \(V(q_\Bbbk)\) in Figure
\ref{fig:MaxRestIrr},
where \(H\) denotes a hyperplane different from \(V(3)\) and \(L=V(c)\) denotes the limit space:

\begin{figure}[ht]
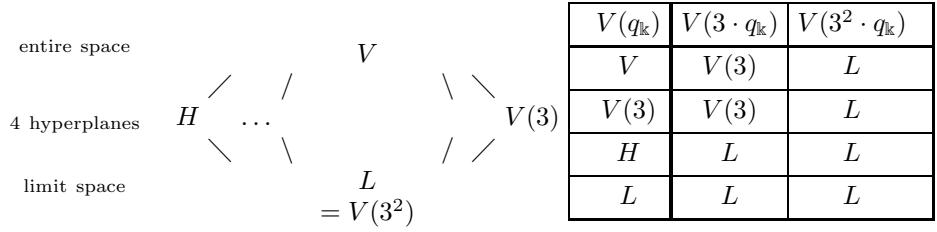

\caption{\(3\)-ring spaces \(V(f)\), \(f\mid c\), for irregular \(c\), \(\delta(3)=1\), \(\delta(3^2)=2\)}
\label{fig:MaxRestIrr}

$$
\vbox{
\offinterlineskip
\halign{
\strut\hfil$#$\hfil\quad &\ \hfil$#$\hfil\ &\hfil$#$\hfil &\ \hfil$#$\hfil\ &\hfil$#$\hfil &\quad\hfil$#$\hfil\quad &\hfil$#$\hfil &\ \hfil$#$\hfil\ &\hfil$#$\hfil &\ \hfil$#$\hfil\ \cr
{}^{\rm entire\ space}   &                 &              &                 &              & V                      &              &                 &              &               \cr
                         &                 & \diagup      &                 & /            &                        & \backslash   &                 & \diagdown    &               \cr
{}_{4\ {\rm hyperplanes}}& H               &              & \ldots          &              &                        &              &                 &              & V(3)          \cr
                         &                 & \diagdown    &                 & \backslash   &                        & /            &                 & \diagup      &               \cr
{}_{\rm limit\ space}    &                 &              &                 &              &  L                     &              &                 &              &               \cr
                         &                 &              &                 &              & =V(3^2)                &              &                 &              &               \cr
}}
\vbox{
\offinterlineskip
\halign{
 \vrule # & \strut\quad\hfill $#$\hfill \ & \vrule\vrule # & \ \hfill $#$\hfill \  & \vrule # &\ \hfill $#$\hfill \quad & \vrule # \cr
\noalign{\hrule}
          &       \omit                   & height 2pt     &                       &          &                         &          \cr
          & V(q_\Bbbk)                    &                & V(3\cdot q_\Bbbk)     &          & V(3^2\cdot q_\Bbbk)     &          \cr
          &       \omit                   & height 2pt     &                       &          &                         &          \cr
\noalign{\hrule}
          &       \omit                   & height 2pt     &                       &          &                         &          \cr
          & V                             &                & V(3)                  &          & L                       &          \cr
          &       \omit                   & height 2pt     &                       &          &                         &          \cr
\noalign{\hrule}
          &       \omit                   & height 2pt     &                       &          &                         &          \cr
          & V(3)                          &                & V(3)                  &          & L                       &          \cr
          &       \omit                   & height 2pt     &                       &          &                         &          \cr
\noalign{\hrule}
          &       \omit                   & height 2pt     &                       &          &                         &          \cr
          & H                             &                & L                     &          & L                       &          \cr     
          &       \omit                   & height 2pt     &                       &          &                         &          \cr
\noalign{\hrule}
          &       \omit                   & height 2pt     &                       &          &                         &          \cr
          & L                             &                & L                     &          & L                       &          \cr
          &       \omit                   & height 2pt     &                       &          &                         &          \cr
\noalign{\hrule}
}
}
$$

\end{figure}

\noindent
The first and second row in the table of Figure
\ref{fig:MaxRestIrr}
do not give a contribution.
Actually, only the value \(\delta=2\) occurs and thus,
according to equation (5.4.2) and Prop.
\ref{prp:Combinatorics},
\[L_s(2)=\frac{1}{3}\sum_{\delta=0}^2\,3^{2-\delta}\cdot\#\lbrace\Bbbk\mid\delta(9q_\Bbbk)=\delta(3q_\Bbbk)=\delta\rbrace\]
\[=\frac{1}{3}(\#\lbrace\Bbbk\mid\delta(q_\Bbbk)=1,\ V(q_\Bbbk)\ne V(3)\rbrace
+\#\lbrace\Bbbk\mid\delta(q_\Bbbk)=2\rbrace)\]
\[=\frac{1}{3}(\sum_{i\ne i_0}\,\left\lbrack\binom{u+n_1(H_i)}{s}-\binom{u}{s}\right\rbrack
+\binom{t}{s}-\sum_{i=1}^4\,\left\lbrack\binom{u+n_1(H_i)}{s}-\binom{u}{s}\right\rbrack-\binom{u}{s})\]
\[=\frac{1}{3}(\binom{t}{s}-\binom{u+n_1(V(3))}{s},\]
since \(\binom{u}{s}\) occurs \(-3+4-1=0\) times.
Note that partially \(\delta(q_\Bbbk)<\delta(3q_\Bbbk)=\delta(9q_\Bbbk)\).

\end{enumerate}

\end{proof}

\noindent
Now we finish the proofs of our main theorems
by inserting the simplified inner sums of Proposition
\ref{prp:InnerSum}
into the outer sums of the multiplicity formulas (5.3.1) and (5.4.1),
and using the binomial formula,
taking into consideration that \(\binom{N}{s}=0\) for \(s>N\).

\begin{proof}[Proof of Theorem
\ref{thm:Codim2Reg}]

Inserting the inner sum of Proposition
\ref{prp:InnerSum}, (1),
into formula (5.3.1),
and inserting the inner sums of Proposition
\ref{prp:InnerSum}, (2) and (3),
into formula (5.4.1)
leads to the same non-degenerate multiplicity formula (4.1.1).

(1)
First, we insert the inner sum for a regular conductor \(c\) with \(\delta_p(c)=2\),
\[L_s(2)=\sum_{i=1}^{p+1}\,\binom{u+n_1(H_i)}{s},\]
into the outer sum of equation (5.3.1),
\[
m_p(d,c)
=p^{\varrho}\cdot\frac{1}{p^{\delta(c)}}\cdot
\left\lbrack(p-1)^{\tau-1}+\sum_{s=0}^{\tau}\,(-1)^{\tau-s}\cdot p^s\cdot L_s(\delta(c))\right\rbrack,
\]
briefly writing \(n_i=n_1(H_i)\):
\[
\sum_{s=0}^{\tau}\,(-1)^{\tau-s}\cdot p^s\cdot L_s(2)
=\sum_{s=0}^{\tau}\,(-1)^{\tau-s}\cdot p^s\cdot\sum_{i=1}^{p+1}\,\binom{u+n_i}{s}\]
\[=\sum_{i=1}^{p+1}\,(-1)^{\tau-u-n_i}\cdot\sum_{s=0}^{u+n_i}\,\binom{u+n_i}{s}p^s(-1)^{u+n_i-s}
=\sum_{i=1}^{p+1}\,(-1)^{\tau-u-n_i}\cdot (p-1)^{u+n_i}.
\]
The complete expression on the right hand side of equation (5.3.1) is:
\[
m_p(d,c)
=p^{\varrho}\cdot\frac{1}{p^{2}}\cdot
\left\lbrack(p-1)^{\tau-1}+\sum_{i=1}^{p+1}\,(-1)^{\tau-u-n_i}\cdot (p-1)^{u+n_i}\right\rbrack\]
\[=p^{\varrho+\omega}\cdot(p-1)^{u}\cdot\frac{1}{p^{2}}\cdot
\left\lbrack(p-1)^{v-1}+\sum_{i=1}^{p+1}\,(-1)^{v-n_i}\cdot (p-1)^{n_i}\right\rbrack,
\]
since \(\tau-u=v\), and we have \(\omega=0\) for a regular conductor \(c\).

(2)
Second, we insert the inner sum for a freely irregular conductor \(c\) with \(\delta_3(c)=2\),
\[L_s(2)=\frac{1}{3}\left\lbrack\binom{t}{s}+2\cdot\sum_{i=1}^{4}\,\binom{\tilde{u}+n_1(H_i)}{s}\right\rbrack,\]
into equation (5.4.1), briefly writing \(n_i=n_1(H_i)\):
\[
m_3(d,c)
=3^{\varrho}\cdot\sum_{s=0}^{t}\,(-1)^{t-s}\cdot 3^s\cdot L_s(2)
=3^{\varrho}\cdot\sum_{s=0}^{t}\,(-1)^{t-s}\cdot 3^s\cdot
\frac{3}{3^2}\left\lbrack\binom{t}{s}+2\cdot\sum_{i=1}^{4}\,\binom{\tilde{u}+n_i}{s}\right\rbrack\]
\[=3^{\varrho+1}\cdot\frac{1}{3^2}
\left\lbrack\sum_{s=0}^{t}\,\binom{t}{s}\cdot 3^s\cdot (-1)^{t-s}
+2\cdot\sum_{i=1}^{4}\,(-1)^{t-\tilde{u}-n_i}\sum_{s=0}^{\tilde{u}+n_i}\,\binom{\tilde{u}+n_i}{s}\cdot 3^s\cdot (-1)^{\tilde{u}+n_i-s}\right\rbrack\]
\[=3^{\varrho+1}\cdot\frac{1}{3^2}
\left\lbrack (3-1)^{t}
+2\cdot\sum_{i=1}^{4}\,(-1)^{t-\tilde{u}-n_i}\cdot (3-1)^{\tilde{u}+n_i}\right\rbrack
=3^{\varrho+1}\cdot 2^{\tilde{u}+1}\cdot\frac{1}{3^2}
\left\lbrack 2^{t-\tilde{u}-1}
+\sum_{i=1}^{4}\,(-1)^{t-\tilde{u}-n_i}\cdot 2^{n_i}\right\rbrack\]
\[=3^{\varrho+\omega}\cdot 2^{u}\cdot\frac{1}{3^2}
\left\lbrack 2^{v-1}
+\sum_{i=1}^{4}\,(-1)^{v-n_i}\cdot 2^{n_i}\right\rbrack,
\]
since \(\omega=1\) for an irregular conductor \(c\), \(u=\tilde{u}+1\), and \(t-\tilde{u}=t+1-\tilde{u}-1=\tau-u=v\).

(3)
Third, we insert the inner sum for a tamely irregular conductor \(c\) with \(\delta_3(c)=2\),
\[L_s(2)=\frac{1}{3}\left\lbrack\binom{t}{s}+3\cdot\binom{u+\tilde{n}_1(H_{i_0})}{s}
-\sum_{i=1}^{4}\,\binom{u+\tilde{n}_1(H_i)}{s}\right\rbrack,\]
into equation (5.4.1), briefly writing \(\tilde{n}_i=\tilde{n}_1(H_i)\), and recalling that \(V_3(3^2)=H_{i_0}\):
\[
m_3(d,c)
=3^{\varrho}\cdot\sum_{s=0}^{t}\,(-1)^{t-s}\cdot 3^s\cdot L_s(2)
=3^{\varrho}\cdot\sum_{s=0}^{t}\,(-1)^{t-s}\cdot 3^s\cdot
\frac{3}{3^2}\left\lbrack\binom{t}{s}+3\cdot\binom{u+\tilde{n}_{i_0}}{s}
-\sum_{i=1}^{4}\,\binom{u+\tilde{n}_i}{s}\right\rbrack\]
\[=3^{\varrho+1}\cdot\frac{1}{3^2}
\left\lbrack\sum_{s=0}^{t}\,\binom{t}{s}\cdot 3^s\cdot (-1)^{t-s}
+3\cdot (-1)^{t-u-\tilde{n}_{i_0}}\sum_{s=0}^{u+\tilde{n}_{i_0}}\,\binom{u+\tilde{n}_{i_0}}{s}\cdot 3^s\cdot (-1)^{u+\tilde{n}_{i_0}-s}\right.\]
\[\left.-\sum_{i=1}^{4}\,(-1)^{t-u-\tilde{n}_i}\sum_{s=0}^{u+\tilde{n}_i}\,\binom{u+\tilde{n}_i}{s}\cdot 3^s\cdot (-1)^{u+\tilde{n}_i-s}\right\rbrack\]
\[=3^{\varrho+1}\cdot\frac{1}{3^2}
\left\lbrack (3-1)^{t}
+3\cdot (-1)^{t-u-\tilde{n}_{i_0}}\cdot (3-1)^{u+\tilde{n}_{i_0}}
-\sum_{i=1}^{4}\,(-1)^{t-u-\tilde{n}_i}\cdot (3-1)^{u+\tilde{n}_i}\right\rbrack\]
\[=3^{\varrho+1}\cdot 2^{u}\cdot\frac{1}{3^2}
\left\lbrack 2^{t-u}
+3\cdot (-1)^{t-u-\tilde{n}_{i_0}}\cdot 2^{\tilde{n}_{i_0}}
-(-1)^{t-u-\tilde{n}_{i_0}}\cdot 2^{\tilde{n}_{i_0}}+\sum_{i\ne i_0}\,(-1)^{t+1-u-\tilde{n}_i}\cdot 2^{\tilde{n}_i}\right\rbrack\]
\[=3^{\varrho+\omega}\cdot 2^{u}\cdot\frac{1}{3^2}
\left\lbrack 2^{v-1}
+\sum_{i=1}^{4}\,(-1)^{v-n_i}\cdot 2^{n_i}\right\rbrack,
\]
since \(\omega=1\) for an irregular conductor \(c\), \(t-u=t+1-u-1=\tau-u-1=v-1\),
and \(n_{i_0}=\tilde{n}_{i_0}+1\), but \(n_i=\tilde{n}_i\) for \(i\ne i_0\).

\end{proof}

\begin{proof}[Proof of Theorem
\ref{thm:Codim2Irr}]

Inserting the degenerate inner sums of Proposition
\ref{prp:InnerSum}, (4) and (5),
into formula (5.4.1) gives the degenerate multiplicity formula (4.2.1).

(1)
First, we insert the inner sum (4) for a wildly irregular conductor \(c\) with \(\delta_3(c)=2\),
\[L_s(2)=\frac{1}{3}\left\lbrack\binom{t}{s}+2\cdot\binom{u+\tilde{n}_1(H_{i_0})}{s}\right\rbrack,\]
into equation (5.4.1), briefly writing \(\tilde{n}_i=\tilde{n}_1(H_i)\), and recalling that \(V_3(3)=H_{i_0}\):
\[
m_3(d,c)
=3^{\varrho}\cdot\sum_{s=0}^{t}\,(-1)^{t-s}\cdot 3^s\cdot L_s(2)
=3^{\varrho}\cdot\sum_{s=0}^{t}\,(-1)^{t-s}\cdot 3^s\cdot
\frac{1}{3}\left\lbrack\binom{t}{s}+2\cdot\binom{u+\tilde{n}_{i_0}}{s}\right\rbrack\]
\[=3^{\varrho}\cdot\frac{1}{3}\left\lbrack\sum_{s=0}^{t}\,\binom{t}{s}\cdot 3^s\cdot (-1)^{t-s}
+2\cdot(-1)^{t-u-\tilde{n}_{i_0}}\sum_{s=0}^{u+\tilde{n}_{i_0}}\,\binom{u+\tilde{n}_{i_0}}{s}\cdot 3^s\cdot (-1)^{u+\tilde{n}_{i_0}-s}\right\rbrack\]
\[=3^{\varrho}\cdot\frac{1}{3}\left\lbrack (3-1)^{t}+2\cdot(-1)^{t-u-\tilde{n}_{i_0}}\cdot (3-1)^{u+\tilde{n}_{i_0}}\right\rbrack
=3^{\varrho}\cdot 2^{u+\tilde{n}_{i_0}+1}\cdot\frac{1}{3}\left\lbrack 2^{t-u-\tilde{n}_{i_0}-1}-(-1)^{t-u-\tilde{n}_{i_0}-1}\right\rbrack\]
\[=3^{\varrho}\cdot 2^{u+n_{i_0}}\cdot\frac{1}{3}\left\lbrack 2^{v-n_{i_0}-1}-(-1)^{v-n_{i_0}-1}\right\rbrack,
\]
since \(n_{i_0}=\tilde{n}_{i_0}+1\) and \(t-u= t+1-u-1=\tau-u-1=v-1\).

(2)
Second, we insert the inner sum (5) for a restrictively irregular conductor \(c\) with \(\delta_3(c)=2\),
\[L_s(2)=\frac{1}{3}\left\lbrack\binom{t}{s}-\binom{u+n_1(H_{i_0})}{s}\right\rbrack,\]
into equation (5.4.1), briefly writing \(n_i=n_1(H_i)\), and recalling that \(V_3(3)=H_{i_0}\):
\[
m_3(d,c)
=3^{\varrho}\cdot\sum_{s=0}^{t}\,(-1)^{t-s}\cdot 3^s\cdot L_s(2)
=3^{\varrho}\cdot\sum_{s=0}^{t}\,(-1)^{t-s}\cdot 3^s\cdot
\frac{1}{3}\left\lbrack\binom{t}{s}-\binom{u+n_{i_0}}{s}\right\rbrack\]
\[=3^{\varrho}\cdot\frac{1}{3}\left\lbrack\sum_{s=0}^{t}\,\binom{t}{s}\cdot 3^s\cdot (-1)^{t-s}
-(-1)^{t-u-n_{i_0}}\sum_{s=0}^{u+n_{i_0}}\,\binom{u+n_{i_0}}{s}\cdot 3^s\cdot (-1)^{u+n_{i_0}-s}\right\rbrack\]
\[=3^{\varrho}\cdot\frac{1}{3}\left\lbrack (3-1)^{t}-(-1)^{t-u-n_{i_0}}\cdot (3-1)^{u+n_{i_0}}\right\rbrack
=3^{\varrho}\cdot 2^{u+n_{i_0}}\cdot\frac{1}{3}\left\lbrack 2^{t-u-n_{i_0}}-(-1)^{t-u-n_{i_0}}\right\rbrack\]
\[=3^{\varrho}\cdot 2^{u+n_{i_0}}\cdot\frac{1}{3}\left\lbrack 2^{v-n_{i_0}-1}-(-1)^{v-n_{i_0}-1}\right\rbrack,
\]
since \(t-u=t+1-u-1=\tau-u-1=v-1\).

\end{proof}



\section{Actual occurrence of \(3\)-defect \(2\)}
\label{s:Codim2Cmpl}

\begin{example}
\label{xpl:Codim2Cmpl}

For a complex quadratic field \(K\) of discriminant \(d<-3\)
and \(3\)-class rank \(\sigma_3=\varrho_3=2\),
we determine two generating principal cubes \(\alpha_1,\alpha_2\) of ideals of \(K\),
such that \(V_3=I_3/K^3=\langle\alpha_1,\alpha_2\rangle\),
by means of
\cite[Algorithm, Step 1--5, pp. 80--81]{Ma},
similarly as in Example
\ref{xpl:Codim1Cmpl}.
That is,
we select two independent forms \(F_1,F_2\) of discriminant \(d_{F_i}=d\) and order \(3\),
calculate the composita \(F_3=F_1F_2\) and \(F_4=F_1F_2^2\),
and determine corresponding algebraic integers
\(\alpha_i=\frac{1}{2}(x_i+y_i\sqrt{d})\), for \(1\le i\le 4\).

A sufficient condition for a \(3\)-admissible conductor \(c\) over \(K\)
to be free with \(3\)-defect \(\delta_3(c)=0\) is that
\(c\mid y_i\) or \(c\mid x_i\), for all \(1\le i\le 4\),
indicated with boldface font in Table
\ref{tbl:Codim2CmplFree}.

\begin{enumerate}

\item
In Table
\ref{tbl:Codim2CmplFree}
we present the discriminants \(d\) with smallest values
of complex quadratic fields \(K\) having \(3\)-class rank \(\varrho_3=2\),
for which the prime conductors \(2\le c\le 5\)
are free with \(3\)-defect \(\delta_3(c)=0\).
According to formula (3.2.1),
where we have to put \(p=3\), \(\varrho=2\), \(\omega=0\), \(\tau=1\),
the corresponding complex cubic fields \(L\) arise in families of multiplicity \(9\)
with discriminant \(d_L=c^2d\).
They occur in the table
of Fung and Williams
\cite{FgWl}
and are discussed in
\cite[Part 1, d), p. S56]{Ma3}.\\
It should be pointed out that in the restrictive cases \(\delta_3(c)=1\)
there is also exactly one \(\alpha_iK(c)^3\in I_3(c)\cap R_cK(c)^3/K(c)^3\),
since otherwise we had \(\delta_3(c)=2\),
which is impossible for a prime conductor,
since the first inequality (2.3.3) enforces \(\delta_3(c)\le t+w=1\).

\renewcommand{\arraystretch}{1.0}

\begin{table}[ht]
\caption{First occurrence of free prime conductors \(2\le c\le 5\) for \(d<-3\), \(\varrho_3=2\)}
\label{tbl:Codim2CmplFree}
\begin{center}
\begin{tabular}{|r|l|c|r|c||c|c|c|}
\hline
 \(d\)         & condition             & \(F_i=(A,B,C)\)    & \(r_i\)   &   \((x_i,y_i)\)              & \(c\)  & \(d_L=c^2d\)   & \(m_3(d,c)\) \\
\hline
 \(-103\,627\) & \(\equiv 5\pmod{8}\)  &    \((47,33,557)\) &   \(571\) &    \((23\,704,\mathbf{42})\) &  \(2\) &  \(-414\,508\) &        \(9\) \\
               &                       &    \((67,25,389)\) &    \(67\) &         \((888,\mathbf{2})\) &        &                &              \\
               &                       & \((139,-117,211)\) &   \(139\) &      \((2\,648,\mathbf{6})\) &        &                &              \\
               &                       &   \((77,-13,337)\) &   \(337\) &    \((11\,250,\mathbf{16})\) &        &                &              \\
\hline
  \(-96\,027\) & \(\equiv +3\pmod{9}\) &    \((61,29,397)\) &    \(61\) &         \((209,\mathbf{3})\) &  \(3\) &  \(-864\,243\) &        \(9\) \\
               &                       &   \((103,81,249)\) &   \(103\) &         \((956,\mathbf{6})\) &        &                &              \\
               &                       &  \((123,117,223)\) &   \(223\) &      \((6\,595,\mathbf{3})\) &        &                &              \\
               &                       & \((157,-123,177)\) &   \(157\) &      \((3\,823,\mathbf{3})\) &        &                &              \\
\hline
  \(-12\,067\) & \((d/5)=-1\)          &    \((23,13,133)\) &   \(199\) &     \((1\,164,\mathbf{50})\) &  \(5\) &  \(-301\,675\) &        \(9\) \\
               &                       &     \((47,23,67)\) &    \(67\) &         \((\mathbf{475},9)\) &        &                &              \\
               &                       &    \((49,-43,71)\) &   \(163\) &     \((\mathbf{2\,045},33)\) &        &                &              \\
               &                       &    \((53,-21,59)\) &   \(229\) &      \((6\,909,\mathbf{5})\) &        &                &              \\
\hline
\end{tabular}
\end{center}
\end{table}

\item
Concerning the actual occurrence of complex cubic fields \(L\)
with discriminant \(d_L=c^2d\) and \(3\)-defect \(\delta_3(c)=2\),
the smallest case we found (not necessarily minimal) is presented in Table
\ref{tbl:Codim2CmplRstr}.

\renewcommand{\arraystretch}{1.0}

\begin{table}[ht]
\caption{Small example of a restrictive conductor \(c\) with \(\delta_3(c)=2\)}
\label{tbl:Codim2CmplRstr}
\begin{center}
\begin{tabular}{|c|l|c|r|c||c|c|c|}
\hline
 \(d\)       & conditions            & \(F_i=(A,B,C)\)  & \(r_i\) & \((x_i,y_i)\) & \(c\)  & \(d_L=c^2d\)      & \(m_3(d,c)\) \\
\hline
 \(-4\,027\) & \(\equiv 5\pmod{8}\)  &    \((13,9,79)\) &  \(13\) &    \((69,1)\) & \(90\) & \(-32\,618\,700\) &        \(9\) \\
             & \((d/5)=-1\)          &   \((17,11,61)\) &  \(61\) &   \((43,15)\) &        &                   &              \\
             & \(\equiv -1\pmod{3}\) &    \((19,1,53)\) &  \(19\) &   \((153,1)\) &        &                   &              \\
             &                       &  \((29,-27,41)\) &  \(43\) &   \((416,6)\) &        &                   &              \\
\hline
\end{tabular}
\end{center}
\end{table}

\noindent
It is obvious that
\(\alpha_2\in\mathcal{O}_5\), \(\alpha_3\in\mathcal{O}_9\), \(\alpha_4\in\mathcal{O}_2\),
and thus the occupation numbers of the four hyperplanes \(H_i<V_3\) in Figure
\ref{fig:Codim2CmplRstr}
with respect to the restrictive \(3\)-admissible conductor \(c=3^2\cdot 2\cdot 5=90\)
are given by \((n_1,n_2,n_3,n_4)=(0,1,1,1)\),
whence formula (4.1.1) yields the \(3\)-multiplicity \(m_3(-4\,027,90)=9\).
This has been confirmed by Belabas using his CCF algorithm
\cite{Be1,Be2}
and PARI
\cite{PARI}
in personal communication.


\begin{figure}[ht]
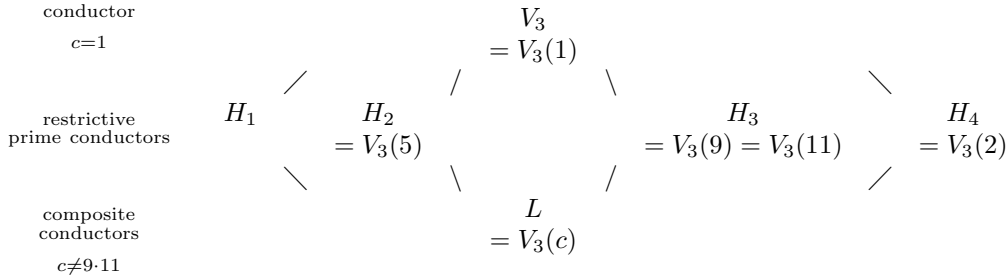

\caption{Positions of \(3\)-ring spaces \(V_3(c)\), \(c\mid 990\), over \(K=\mathbb{Q}(\sqrt{-4\,027})\)}
\label{fig:Codim2CmplRstr}

$$
\vbox{
\offinterlineskip
\halign{
\strut\quad\hfil$#$\hfil\quad  & \quad\hfil$#$\hfil\quad &\hfil$#$\hfil &\quad\hfil$#$\hfil\quad &\hfil$#$\hfil &\quad\hfil$#$\hfil\quad &\hfil$#$\hfil &\quad\hfil$#$\hfil\quad &\hfil$#$\hfil &\quad\hfil$#$\hfil\quad \cr
 {}^{\rm conductor}            &                         &              &                        &              & V_3                    &              &                        &              &                        \cr
 {}^{c=1}                      &                         &              &                        &              & =V_3(1)                &              &                        &              &                        \cr
                               &                         & \diagup      &                        &      /       &                        & \backslash   &                        & \diagdown    &                        \cr
{}_{\rm restrictive}           & H_1                     &              & H_2                    &              &                        &              & H_3                    &              & H_4                    \cr
{}^{\rm prime\ conductors}     &                         &              &       = V_3(5)         &              &                        &              &   = V_3(9) = V_3(11)  &              &        = V_3(2)        \cr
                               &                         & \diagdown    &                        & \backslash   &                        &      /       &                        & \diagup      &                        \cr
 {}_{\rm composite}            &                         &              &                        &              &         L              &              &                        &              &                        \cr
 {}^{\rm conductors}           &                         &              &                        &              &        = V_3(c)        &              &                        &              &                        \cr
 {}^{c\ne 9\cdot 11}           &                         &              &                        &              &                        &              &                        &              &                        \cr
}}                                                                                                                                              
$$

\end{figure}

\noindent
It is illuminating to compare \(3\)-multiplicities \(m_3(d,c)\) of various conductors \(c\) over \(d=-4\,027\) in Table
\ref{tbl:ComparingConductors},
using the order of \((n_1,n_2,n_3,n_4)\) in Table
\ref{tbl:RstrMult}.

\renewcommand{\arraystretch}{1.0}

\begin{table}[ht]
\caption{Various conductors \(c\) dividing \(990\) over \(K=\mathbb{Q}(\sqrt{-4\,027})\)}
\label{tbl:ComparingConductors}
$$
\vbox{
\offinterlineskip
\halign{
 \vrule # & \strut\ \hfil$#$\hfil\ &   \vrule # & \ $#$\   &\ \hfil$#$\hfil\ & \vrule # &\ \hfil$#$\hfil\  & \vrule # & \ \hfil$#$\hfil\     & \vrule # & \ \hfil$#$   \    & \vrule # \cr
\noalign{\hrule}
          &        c               &            & \tau     & \delta_3(c) &          &   m_3(d,c)           &          & (n_1,\ldots,n_4)     &          &      d_L=c^2d     &          \cr
          &   \omit                & height 2pt &          &             &          &                      &          &                      &          &                   &          \cr
\noalign{\hrule}
          &   \omit                & height 2pt &          &             &          &                      &          &                      &          &                   &          \cr
          &        1               &            &   0      &      0      &          &          4           &          &  (0,0,0,0)           &          &           -4\,027 &          \cr
          &   \omit                & height 2pt &          &             &          &                      &          &                      &          &                   &          \cr
\noalign{\hrule}
          &   \omit                & height 2pt &          &             &          &                      &          &                      &          &                   &          \cr
          & 2, 5, 3^2, 11          &            &   1      &      1      &          &          0           &          &  (1,0,0,0)           &          &                   &          \cr
 & 2\cdot 5,2\cdot 3^2,2\cdot 11,5\cdot 3^2,5\cdot 11 & & 2 &     2      &          &          0           &          &  (1,1,0,0)           &          &                   &          \cr
          &   \omit                & height 2pt &          &             &          &                      &          &                      &          &                   &          \cr
\noalign{\hrule}
          &   \omit                & height 2pt &          &             &          &                      &          &                      &          &                   &          \cr
          & 3^2\cdot 11            &            &   2      &      1      &          &          9           &          &  (2,0,0,0)           &          &    -39\,468\,627  &          \cr
          &   \omit                & height 2pt &          &             &          &                      &          &                      &          &                   &          \cr
\noalign{\hrule}
          &   \omit                & height 2pt &          &             &          &                      &          &                      &          &                   &          \cr
          & 2\cdot 3^2\cdot 11,5\cdot 3^2\cdot 11 & &   3  &      2      &          &          0           &          &  (2,1,0,0)           &          &                   &          \cr
          &   \omit                & height 2pt &          &             &          &                      &          &                      &          &                   &          \cr
\noalign{\hrule}
          &   \omit                & height 2pt &          &             &          &                      &          &                      &          &                   &          \cr
          & 2\cdot 5\cdot 3^2      &            &   3      &      2      &          &          9           &          &  (1,1,1,0)           &          &    -32\,618\,700  &          \cr
          & 2\cdot 5\cdot 11       &            &   3      &      2      &          &          9           &          &  (1,1,1,0)           &          &    -48\,726\,700  &          \cr
          &   \omit                & height 2pt &          &             &          &                      &          &                      &          &                   &          \cr
\noalign{\hrule}
          &   \omit                & height 2pt &          &             &          &                      &          &                      &          &                   &          \cr
          & 2\cdot 5\cdot 3^2\cdot 11 &         &   4      &      2      &          &          9           &          &  (2,1,1,0)           &          & -3\,946\,862\,700 &          \cr
\noalign{\hrule}
}
}
$$

\end{table}

\item
We conclude by a survey of minimal discriminants for the occurrence of
an irregular prime power conductor \(c=3^2\) in the case \(p=3\), \(d\equiv -3\pmod{9}\)
in Table
\ref{tbl:IrrPrmPwr}.
Some results for \(\varrho_3=2\) are due to personal communication by Belabas,
who used his CCF algorithm 
\cite{Be1,Be2}
and PARI
\cite{PARI}
to determine
the discriminants \(-42\,591\), \(-128\,451\), and \(-2\,069\,688\).

\renewcommand{\arraystretch}{1.0}

\begin{table}[ht]
\caption{Quadratic fields admitting an irregular conductor \(c=9\)}
\label{tbl:IrrPrmPwr}
\begin{center}
\begin{tabular}{|r||c|c||c|c||c|r|r|}
\hline
 \(d\)            &\(\varrho_3\) & \(\sigma_3\) & \(\delta_3(3)\) & \(\delta_3(9)\) & \(c\) & \(d_L=c^2d\)       & \(m_3(d,c)\) \\
\hline
           \(-3\) &        \(0\) &        \(1\) &           \(1\) &           \(1\) & \(9\) &           \(-243\) &        \(1\) \\
\hline
          \(-39\) &        \(0\) &        \(0\) &           \(0\) &           \(0\) & \(9\) &        \(-3\,159\) &        \(3\) \\
\hline
         \(-255\) &        \(1\) &        \(1\) &           \(1\) &           \(1\) & \(9\) &       \(-20\,655\) &        \(3\) \\
         \(-687\) &        \(1\) &        \(1\) &           \(0\) &           \(1\) & \(9\) &                --- &        \(0\) \\
      \(-3\,387\) &        \(1\) &        \(1\) &           \(0\) &           \(0\) & \(9\) &      \(-274\,347\) &        \(9\) \\
\hline
      \(-8\,751\) &        \(2\) &        \(2\) &           \(1\) &           \(2\) & \(9\) &                --- &        \(0\) \\
     \(-42\,591\) &        \(2\) &        \(2\) &           \(1\) &           \(1\) & \(9\) &   \(-3\,449\,871\) &        \(9\) \\
    \(-128\,451\) &        \(2\) &        \(2\) &           \(0\) &           \(1\) & \(9\) &                --- &        \(0\) \\
 \(-2\,069\,688\) &        \(2\) &        \(2\) &           \(0\) &           \(0\) & \(9\) & \(-167\,644\,728\) &       \(27\) \\
\hline
 \(-4\,447\,704\) &        \(3\) &        \(3\) &           \(1\) &           \(2\) & \(9\) &                --- &        \(0\) \\
\hline
           \(24\) &        \(0\) &        \(1\) &           \(1\) &           \(1\) & \(9\) &         \(1\,944\) &        \(1\) \\
           \(69\) &        \(0\) &        \(1\) &           \(0\) &           \(1\) & \(9\) &                --- &        \(0\) \\
          \(717\) &        \(0\) &        \(1\) &           \(0\) &           \(0\) & \(9\) &        \(58\,077\) &        \(3\) \\
\hline
\end{tabular}
\end{center}
\end{table}

\end{enumerate}

\end{example}



\section{Acknowledgements}
\label{s:Thanks}

We gratefully acknowledge that this research was
supported by the Austrian Science Fund (FWF): P 26008-N25.
Further, we thank the anonymous referees for valuable suggestions
to improve the readability.




\end{document}